\documentclass[11pt, reqno]{amsart}
\usepackage{amsmath, amssymb}
\usepackage{color}
\usepackage{bm}
\usepackage[normalem]{ulem}
\usepackage{graphicx}
\usepackage{tikz}
\usepackage[all]{xy}
\usepackage{mathtools}
\usepackage{enumitem}
\allowdisplaybreaks 
\usepackage{hyperref}

\newtheorem{theorem}{Theorem}[section]
\newtheorem{lemma}[theorem]{Lemma}
\newtheorem{proposition}[theorem]{Proposition}
\newtheorem{corollary}[theorem]{Corollary}

\theoremstyle{definition}

\newtheorem{example}[theorem]{Example}

\theoremstyle{remark}
\newtheorem{remark}[theorem]{Remark}
\numberwithin{equation}{section}

\newcommand{\Z}{\mathbb{Z}}
\newcommand{\Q}{\mathbb{Q}}
\newcommand{\C}{\mathbb{C}}
\newcommand{\K}{\mathbb{K}}
\newcommand{\g}{\mathfrak{g}}
\newcommand{\A}{\mathcal{A}}
\newcommand{\I}{\mathcal{I}}
\newcommand{\T}{\mathcal{T}}

\newcommand{\Ker}{\operatorname{Ker}}
\renewcommand{\Im}{\operatorname{Im}}
\newcommand{\Isom}{\operatorname{Isom}}
\newcommand{\gl}{\mathfrak{gl}}
\renewcommand{\sl}{\mathfrak{sl}}
\newcommand{\so}{\mathfrak{so}}
\renewcommand{\sp}{\mathfrak{sp}}
\newcommand{\id}{\mathrm{id}}
\newcommand{\pr}{\mathrm{pr}}

\newcommand{\tr}{\operatorname{tr}}
\newcommand{\KZ}{\mathrm{KZ}}

\newcommand{\ad}{\operatorname{ad}}

\newcommand{\GT}{\mathit{GT}}
\newcommand{\GRT}{\mathit{GRT}}
\newcommand{\grt}{\mathfrak{grt}}
\newcommand{\SL}{\mathcal{SL}}
\newcommand{\St}{\mathrm{St}}
\newcommand{\End}{\mathrm{End}}
\newcommand{\even}{\mathrm{even}}
\newcommand{\uni}{\mathrm{uni}}

\newcommand{\pre}{\mathrm{pre}}
\newcommand{\hor}{\mathrm{hor}}
\newcommand{\Spec}{\mathrm{Spec}}
\newcommand{\PB}{\mathit{PB}}

\newcommand{\chordA}{
\begin{tikzpicture}[scale=0.4, baseline={(0,0.3)}, densely dashed]
    \draw[solid,->] (0,2) -- (0,0);
    \draw[solid,->] (2,2) -- (2,0);
    \draw[solid,->] (4,2) -- (4,0);
    \draw (0,1) -- (2,1);
\end{tikzpicture}
}

\newcommand{\chordB}{
\begin{tikzpicture}[scale=0.4, baseline={(0,0.3)}, densely dashed]
    \draw[solid,->] (0,2) -- (0,0);
    \draw[solid,->] (2,2) -- (2,0);
    \draw[solid,->] (4,2) -- (4,0);
    \draw (2,1) -- (4,1);
\end{tikzpicture}
}

\newcommand{\onechord}{
\begin{tikzpicture}[scale=0.4, baseline={(0,0.3)}, densely dashed]
    \draw[solid,->] (0,2) -- (0,0);
    \draw[solid,->] (2,2) -- (2,0);
    \draw (0,1) -- (2,1);
\end{tikzpicture}
}

\newcommand{\onechordcross}{
\begin{tikzpicture}[scale=0.4, baseline={(0,0.3)}]
    \draw (0,3) -- (2,2);
    \draw (2,3) -- (0,2);
    \draw[->] (0,2) -- (0,0);
    \draw[->] (2,2) -- (2,0);
    \draw[densely dashed] (0,1) -- (2,1);
\end{tikzpicture}
}

\newcommand{\phichord}{
\begin{tikzpicture}[scale=0.4, baseline={(0,0.3)}, densely dashed]
    \draw[solid,->] (0,2) -- (0,0);
    \draw[solid,->] (3,2) -- (3,0);
    \draw (0,1) -- (1,1);
    \draw (3,1) -- (2,1);
    \draw (2,1) arc[start angle=0, end angle=360, radius=0.5];
\end{tikzpicture}
}

\newcommand{\Ugraph}{
\begin{tikzpicture}[scale=0.6, baseline={(0,0.05)}, densely dashed]
\draw[solid, rounded corners=2pt] (0,0.5) -- (0,0.3) -- (0.5,0.3) -- (0.5,0.5);
    \draw[solid, rounded corners=2pt] (0.5,0) -- (0.5,0.2) -- (0, 0.2) -- (0,0);
\end{tikzpicture}
}

\title[]{The Kontsevich invariant and the action of the Grothendieck--Teichm\"{u}ller group on $2$-component string links}

\author{Hisatoshi Kodani}
\address{
Institute of Mathematics for Industry,
Kyushu University \\
744, Motooka, Nishi-ku, Fukuoka, 819-0395\\
Japan}
\email{kodani@imi.kyushu-u.ac.jp}

\author{Yuta Nozaki}
\address{
Faculty of Environment and Information Sciences, Yokohama National University \\
79-7 Tokiwadai, Hodogaya-ku, Yokohama, 240-8501 \\
Japan\vspace{-0.6em}}
\address{
WPI-SKCM$^2$, Hiroshima University \\
1-3-1 Kagamiyama, Higashi-Hiroshima, Hiroshima, 739-8526 \\
Japan}
\email{nozakiy@hiroshima-u.ac.jp}



\begin{document}
\begin{abstract}
The Kontsevich invariant of links is independent of the choice of associator, whereas for tangles this is not the case in general.
In this paper, we focus on $2$-component string links and investigate to what extent the Kontsevich invariant depends on the choice of associator.
As an application, we show that the action of the unipotent part of the Grothendieck--Teichm\"{u}ller group on the algebra of proalgebraic $2$-component string links is non-trivial, which provides a partial answer to a problem posed by Furusho.
\end{abstract}
\maketitle

\setcounter{tocdepth}{1}
\tableofcontents

\section{Introduction}
\label{sec:intro}

Kontsevich invariants are topological invariants for $q$-tangles, which take values in the space of Jacobi diagrams, where a $q$-tangle is a tangle whose endpoints are equipped with non-associative words. 
The invariants were originally introduced by Kontsevich~\cite{Kon93} using Chen's iterated integrals, and then Le and Murakami~\cite{LeMu96CM} gave an alternative combinatorial construction of the invariants via associators. 
See also \cite{Bar97}, \cite{LeMu95TA}, \cite{Piu95}, and \cite{Car93}. 
In general, the Kontsevich invariants depend on the choice of associator, and we write $Z_{\Phi}$ for the Kontsevich invariant defined by an associator $\Phi$.

We now consider the case of $n$-component string links \cite{HL90} as illustrated in Figure~\ref{fig:string_link}, which are also known as homology cylinders over $\Sigma_{0,n+1}$, a surface of genus $0$ with $n+1$ boundary components.
In this case, it is known that the Kontsevich invariants are independent of the choice of associator when $n=1$, but they do depend on it when $n \geq 3$. 
The former is essentially the same as the case for knots. For the latter, as seen in Example~\ref{ex:3-pb_konts_inv}, we can easily observe that there exists a $q$-braid $T$ of three strands such that $Z_{\Phi_\KZ}(T)_{4}\neq Z_{\Phi_\Q^\even}(T)_{4}$.
Here, $Z_{\Phi_\KZ}$ and $Z_{\Phi_\Q^\even}$ denote the Kontsevich invariants associated with the Knizhnik--Zamolodchikov associator $\Phi_\KZ$ and with an even rational horizontal associator $\Phi_\Q^\even$ (see Section~\ref{subsec:associator}), respectively.
Further, the subscript $4$ means their degree $4$ parts.

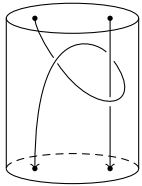
\begin{figure}[h]
\begin{tikzpicture}
\coordinate (a) at (0,0);
\coordinate (b) at (0,-2);
\coordinate (c) at (1,0);
\coordinate (d) at (1,-2);

 \foreach \x in {(a), (b), (c), (d)}{
        \fill \x circle[radius=1pt];
  }
  \begin{scope}[yscale=1.5]{
  \draw(c) -- (1, -0.7);
  \draw[->] (1, -0.78) -- (d);
  \draw (a) .. controls (0, -0.08) and  (0.15, -0.26) .. (0.25,-0.35);
  \draw[->] (0.95, -0.3) .. controls (0.8, -0.2) and (0, 0) .. (b);
  \draw (0.3,-0.4) .. controls (0.8, -0.9) and (1.5, -0.8) .. (1.05, -0.38);
  }
  \end{scope}
 \draw (1.38,0) arc (0:360:0.88cm and 0.2cm);
  \draw[densely dashed] (1.38,-2) arc (0:180:0.88cm and 0.2cm);
   \draw (1.38,-2) arc (0:-180:0.88cm and 0.2cm);
   \draw (-0.38,0) -- (-0.38,-2);
    \draw (1.38,0) -- (1.38,-2);
    \end{tikzpicture}
    \caption{An example of a $2$-component string link.}
    \label{fig:string_link}
\end{figure}

However, when $n=2$, the dependence is somewhat subtle and does not seem to be well understood (cf.~\cite[Problem~3.17]{Fur20}).
As mentioned in the latter part of the introduction, this subtlety is closely connected with a difference between the Grothendieck--Teichm\"uller group action on the proalgebraic pure braids on two strands and that on proalgebraic $2$-component string links (see Section~\ref{subsec:string_link}).

The following proposition states that the Kontsevich invariants for $2$-component string links are independent of the choice of associator up to degree $6$.
\begin{proposition}
\label{prop:deg_6}
Let $T$ be a $2$-component string link and let $\Phi$, $\Phi'$ be associators.
Then, $Z_{\Phi}(T)_{\leq 6} = Z_{\Phi'}(T)_{\leq 6}$.
\end{proposition}

In particular, $Z_{\Phi_\KZ}(T)_{\leq 6} = Z_{\Phi_\Q^\even}(T)_{\leq 6}$.
Regarding higher degree terms, we obtain the following main result.
Let $\K$ be a field of characteristic $0$. 

\begin{theorem}
\label{thm:deg_8}
There exist $2$-component string links $T_n$ $(n\geq 2)$ satisfying the following:
for any associator $\Phi$ with coefficients in $\K$, there exists an associator $\Phi'$ with coefficients in $\K$ such that
\[
Z_{\Phi}(T_n)_{\leq 2n+3}= Z_{\Phi'}(T_n)_{\leq 2n+3}
\text{ and }
Z_{\Phi}(T_n)_{2n+4}\neq Z_{\Phi'}(T_n)_{2n+4}.
\]
Furthermore, if $\Phi$ is horizontal, then $\Phi'$ can also be chosen to be horizontal.
\end{theorem}

Here, a horizontal associator is an associator obtained from a Drinfeld's associator given as non-commutative formal power series in two variables introduced in \cite{Dri91} (see Section~\ref{subsec:associator}).
The proofs of Proposition~\ref{prop:deg_6} and Theorem~\ref{thm:deg_8} will be given in Sections~\ref{subsec:2_arrows} and \ref{subsec:main_theorem}, respectively.

\begin{remark}
\label{rem:deg_8}
For instance, by letting $\K=\C$, Theorem~\ref{thm:deg_8} allows us to compare the Kontsevich invariants associated with an even rational horizontal associator and with the KZ associator.
\end{remark}

As explained above, the Kontsevich invariants of $2$-component string links depend on the choice of the associator $\Phi$.
Nevertheless, one can obtain infinitely many invariants that are independent of $\Phi$ by composing the Kontsevich invariant with weight systems.
Here, a weight system is a certain way to obtain values from Jacobi diagrams using a Lie algebra (see Section~\ref{subsec:Lie_alg_weight}).

\begin{theorem}
\label{thm:weight_system}
    There are infinitely many weight systems $W$ on two strands such that the composition $W \circ Z_{\Phi}$ does not depend on the choice of the associator $\Phi$. 
\end{theorem}

See Theorem~\ref{thm:weight_system_precise} for a more precise statement. It may be worth emphasizing that the above $W \circ Z_{\Phi}$ are invariants not only for $2$-component string links with $q$-structures but also for $2$-component string links themselves since $q$-structures for them are uniquely determined.

These results also have an application to the (motivic) Galois action on proalgebraic string links studied by Furusho in \cite{Fur20} as follows. Let $\GT(\K)$ and $\GRT(\K)$ be the proalgebraic Grothendieck--Teichm\"uller group and the graded Grothendieck--Teichm\"uller group, respectively. Since $\GT(\K)$ and $\GRT(\K)$ act on the set of horizontal assocaitors freely and transitively, the unipotent parts $\GT_1(\K)$ and $\GRT_1(\K)$ also act on the Kontsevich invariants $Z_{\Phi}$ with horizontal associators $\Phi$ by replacements of chosen horizontal associators. 
As is made explicit by Furusho, such actions of $\GT_1(\K)$ and $\GRT_1(\K)$ on $Z_{\Phi}$ are compatible with the $(\GRT(\K), \GT(\K))$-bitorsor structure on the set of associators.
One can define, in a canonical manner, a $\GT(\K)$-action on proalgebraic tangles and a $\GRT(\K)$-action on the corresponding Jacobi diagrams. 
Recall here the fact that there is an embedding of the motivic Galois group of mixed Tate motives of $\Spec(\Z)$ into $\GT(\K)$, 
which is deduced from Brown's result~\cite{Bro12}.
It follows that proalgebraic tangles are endowed with a structure of mixed Tate motives of $\Spec(\Z)$ through the $\GT(\K)$-action.

Furusho further considered the case where tangles are $n$-component strings and showed that the $\GT_1(\K)$-action on proalgebraic $n$-component string links is given by an inner conjugation.
As a result, one may see that the $\GT_1(\K)$-action on proalgebraic $n$-component string links is faithful for $n \geq 3$ and trivial for $n=1$.
For the case where $n=2$, the action of $\GT_1(\K)$ on proalgebraic $2$-component string links remains unclear, although we know that its action on the proalgebraic pure braids on two strands is trivial.

As an application of Theorem~\ref{thm:deg_8}, we can prove Theorem~\ref{thm:GT_1_2-string_link} in Section~\ref{subsec:GT-action}, which gives a partial answer to \cite[Problem~3.17]{Fur20} asking whether the $\GT_1(\K)$-action is faithful or not.

\begin{theorem}
\label{thm:GT_1_2-string_link}
The action of the unipotent part $\GT_1(\K)$ of $\GT(\K)$ on proalgebraic $2$-component string links is non-trivial. 
\end{theorem}

Theorem~\ref{thm:GT_1_2-string_link} tells us that the $\GT_1(\K)$-action on proalgebraic $2$-component string links behaves completely differently from that on proalgebraic $1$-component string links (equivalently, proalgebraic knots).
The latter action is known to be trivial, which decomposes the proalgebraic knots into a direct sum of (infinitely many) pure Tate motives over $\Z$.
The current situation towards \cite[Problem~3.17]{Fur20} is summarized in Table~\ref{tab:GT_action_string_links}.

\begin{table}[h]
\centering
\begin{tabular}{|l||c|c|c|} \hline
 & $n =1$ & $n=2$ & $n \geq 3$\\ \hline \hline
kernel $\Ker \rho$  & $\Ker \rho =\GT_1(\K)$  & $\Ker \rho \subsetneq  \GT_1(\K)$ & $\Ker \rho = \{e\}$ \\ 
 & (\cite{Fur20}) & (Theorem~\ref{thm:GT_1_2-string_link}) & (\cite{Fur20}) \\ \hline
faithfulness & far from faithful & ? & faithful \\ \hline
\end{tabular}
\caption{The kernel and faithfulness of the $\GT(\K)$-action $\rho$ on the algebra of proalgebraic $n$-component string links.}
\label{tab:GT_action_string_links}
\end{table}

From Theorem~\ref{thm:GT_1_2-string_link},
we can also expect that a conjectural action of the profinite Grothendieck--Teichm\"uller group $\widehat{GT}$ on profinite $2$-component string links would not factor through $\widehat{\Z}^{\times}$ (cf.~\cite[Remark~2.44]{Fur17} where a similar question is mentioned for profinite knots). 
In contrast, the action of $\widehat{GT}$ on the profinite pure braid group with two strands is known to factor through $\widehat{\Z}^{\times}$ (cf.~\cite{Ih91}, \cite{Fur17}).
Recently, related work has appeared from the operadic viewpoint~(\cite{RoSi25}, \cite{Wil24}). 

We conclude this introduction with a remark on the original question posed by Furusho (\cite[Problem~3.17]{Fur20}). This question can be reformulated as follows: For any distinct horizontal associators $\Phi$ and $\Phi'$, does $Z_{\Phi} \neq Z_{\Phi'}$ hold as maps on proalgebraic $2$-string links with the unique $q$-structure? It is still open, but, as a byproduct of our work, we find that a generalized version of the question does not hold due to the following proposition shown in Section~\ref{subsec:main_theorem}.

\begin{proposition}
\label{prop:general_replacement}
There are two distinct associators $\Phi$ and $\Phi'$ such that at least one of the two is not horizontal and that $Z_{\Phi}(T) = Z_{\Phi'}(T)$ holds for any $2$-component string link $T$.
\end{proposition}

\subsection*{Acknowledgments}
The authors would like to thank Dror Bar-Natan for his suggestion on Proposition~\ref{prop:J3J2n+1}.
They also express their gratitude to \mbox{Katsumi} Ishikawa for his help with Lemma~\ref{lem:box}.
The authors are grateful to Hidekazu Furusho for his careful reading of the manuscript, and to 
Adrien Brochier, 
Martin Gonzalez, 
Masanobu Kaneko,
Yusuke Kuno,  
Masanori Morishita, 
Hiroaki Nakamura, 
Masatoshi Sato, 
and Yuji Terashima 
for valuable discussions.
This study was supported in part by JSPS KAKENHI Grant Numbers JP23K12974, JP25K06954, and by Institute of Mathematics for Industry, Joint Usage/Research Center in Kyushu University. (FY2025 Workshop(II) ``Topology and Computer 2025'' (2025a038).)

\section{Preliminaries}
\label{sec:prelim}

\subsection{Jacobi diagrams}
\label{subsec:Jacobi}

Let $X$ be a (possibly disconnected) oriented compact $1$-manifold.
A \emph{Jacobi diagram} based on $X$ is a (possibly empty) uni-trivalent graph whose univalent vertices are attached to $X$ and trivalent vertices have cyclic orders.
We use dashed lines for uni-trivalent graphs and solid lines for $X$ as in \cite{Oht02}.
In figures of Jacobi diagrams, the cyclic order of each trivalent vertex is assumed counterclockwise.
We write $\A(X)$ for the $\Q$-vector space generated by Jacobi diagrams based on $X$ subject to the AS, IHX, and STU relations:
\begin{gather*}
\begin{tikzpicture}[scale=0.5, baseline={(0,-0.1)}, densely dashed]
  \draw (0,0) -- (-1,1);
  \draw (0,0) -- (1,1);
  \draw (0,0) -- (0,-1);
\end{tikzpicture}%
+\!\!
\begin{tikzpicture}[scale=0.5, baseline={(0,-0.1)}, densely dashed]
  \coordinate (origin) at (0,0);
  \draw (origin) .. controls +(1,0.5) and +(1,-0.5) .. (-1,1);
  \draw (origin) .. controls +(-1,0.5) and +(-1,-0.5) .. (1,1);
  \draw (origin) -- (0,-1);
\end{tikzpicture}%
\!\!=0,
\qquad
\begin{tikzpicture}[scale=0.5, baseline={(0,-0.1)}, densely dashed]
  \draw (-1,1) -- (1,1) ;
  \draw (-1,-1) -- (1,-1) ;
  \draw (0,1) -- (0,-1) ;
\end{tikzpicture}%
-\ 
\begin{tikzpicture}[scale=0.5, baseline={(0,-0.1)}, densely dashed]
  \draw (-1,1) -- (-1,-1) ;
  \draw (1,1) -- (1,-1) ;
  \draw (-1,0) -- (1,0) ;
\end{tikzpicture}%
\ +
\begin{tikzpicture}[scale=0.5, baseline={(0,-0.1)}, densely dashed]
  \draw (-1,1) -- (1,-1);
  \draw (-1,-1) -- (1,1);
  \draw (-0.5,-0.5) -- (0.5,-0.5);
\end{tikzpicture}%
=0,
\\
\begin{tikzpicture}[scale=0.5, baseline={(0,0.3)}, densely dashed]
\draw[->,solid] (0,0) -- (4,0);
\draw (2,0) -- (2,1) ;
\draw (2,1) -- (1,2) ;
\draw (2,1) -- (3,2) ;
\end{tikzpicture}%
\ -\ 
\begin{tikzpicture}[scale=0.5, baseline={(0,0.3)}, densely dashed]
\draw[->,solid] (0,0) -- (4,0);
\draw (1,0) -- (1,2) ;
\draw (3,0) -- (3,2) ;
\end{tikzpicture}%
\ +\ 
\begin{tikzpicture}[scale=0.5, baseline={(0,0.3)}, densely dashed]
\draw[->,solid] (0,0) -- (4,0);
\draw (1,0) -- (3,2) ;
\draw (3,0) -- (1,2) ;
\end{tikzpicture}%
=0,
\end{gather*}
where the rest of the diagrams are the same in each relation.
Note that we are working with abstract graphs, and what appear to be $4$-valent vertices are not actual vertices.
The \emph{degree} $\deg J$ of a Jacobi diagram $J$ is defined to be half the number of vertices of $J$.
Let $\A_d(X)$ denote the subspace of $\A(X)$ generated by Jacobi diagrams of degree $d$.
The completion $\varprojlim_d \A(X)/\A_{\geq d}(X)$ is denoted by $\widehat{\A}(X)$.
For a bijection $f$ between some endpoints of $X$ and $X'$ compatible with their orientations, we have a linear map $\A(X)\times \A(X')\to \A(X\cup_f X')$.
The image of $(J,J')$ is denoted by $J\cdot J'$.
For a word $\epsilon$ in letters $+$ and $-$, let $\downarrow^{\epsilon}$ denote the disjoint union  of intervals indexed by the letters in $\epsilon$, e.g., ${\downarrow^{+-+}} = {\downarrow\uparrow\downarrow}$.
When $\epsilon = ++\cdots+$ of length $n$, we simply write $\downarrow^{\otimes n}$.
In the case $X = X' = \downarrow^{\epsilon}$, by stacking $X$ on $X'$, the above map makes $\A(\downarrow^{\epsilon})$ and $\widehat{\A}(\downarrow^{\epsilon})$ algebras.
Note that the empty Jacobi diagram on $\downarrow^{\epsilon}$ is the identity element $1$.
For instance, we will discuss the algebra structure of $\A(\downarrow\downarrow)$ in Section~\ref{subsec:2_arrows}.
We write $J\otimes J'$ for the disjoint union of Jacobi diagrams $J$ and $J'$.
For a field $\K$ of characteristic $0$, let $\A(X)_\K = \A(X)\otimes_\Q \K$ and $\widehat{\A}(X)_\K = \widehat{\A}(X)\hat{\otimes}_\Q \K$.

Let us recall the relation between the Drinfeld--Kohno Lie algebra and the algebra of Jacobi diagrams  $\mathcal{A}(\downarrow^{\otimes n})$.  For $n \geq 2$, the \emph{Drinfeld--Kohno Lie algebra} $\mathfrak{t}_n$ is defined as the (completed) graded Lie algebra over $\Q$ generated by $t_{ij}$ $(i\neq j \in \{1, \ldots, n\})$ of degree $1$ subject to the following relations:
\[
t_{ij}=t_{ji},\quad [t_{ij}, t_{kl}]=0\ (\#\{i,j,k,l\}=4), \quad [t_{ij}, t_{ik}+t_{jk}]=0\ (\#\{i,j,k\}=3).
\]
We denote by $U(\mathfrak{t}_n)$ the universal enveloping algebra of $\mathfrak{t}_n$. Then, we have the injective algebra homomorphism
\[
\iota\colon U(\mathfrak{t}_n)\rightarrow \A(\downarrow^{\otimes n})
\]
which sends the generator $t_{ij}$ to the Jacobi diagram consisting of a single dashed edge connecting the $i$th and $j$th strings as follows (\cite{Bar95JKTRV}, \cite{HaMa00}):
\[
\begin{tikzpicture}
\begin{scope}
\node at (-2, 0.6) {$t_{ij}$};
\node at (-1, 0.5) {$\mapsto$};
\foreach \x in {0, 0.5, 1.5, 2, 2.5, 3.5, 4, 4.5, 5.5, 6}
\draw[->,rounded corners=4pt] (\x,1.4) -- (\x, 0);
\foreach \x / \y in {0/1, 0.5/, 1.5/, 2/i, 2.5/, 3.5/, 4/j, 4.5/, 5.5/, 6/n}
\node at (\x, -0.25) {$\y$};
\foreach \x in {1, 3, 5}
\node at (\x, 0.5) {$\cdots$};
\draw[densely dashed,rounded corners=4pt] (2,0.7) -- (4, 0.7);
\end{scope}
\end{tikzpicture}
\]

Let $\A^{\hor}(\downarrow^{\otimes n})\coloneqq\iota(U(\mathfrak{t}_n)) \subset \A(\downarrow^{\otimes n})$.  An element $J \in \A^{\hor}(\downarrow^{\otimes n})$ is called a \emph{horizontal Jacobi diagram}.
\begin{remark} \label{rem:kdalg}
It is known that $\mathfrak{t}_3$ is a central extension of the free Lie algebra with two generators. More precisely, we have
\[
\mathfrak{t}_3 \simeq \mathfrak{f}(t_{12}, t_{23}) \oplus \Q\cdot (t_{12}+t_{23}+t_{13}),
\]
where $\mathfrak{f}(t_{12}, t_{23})$ denotes the free Lie algebra on $t_{12}$ and $t_{23}$ over $\Q$.
\end{remark}

Unless otherwise stated, for $v \in U(\mathfrak{t}_n)$, whenever $v$ is considered as an element of $U(\mathfrak{t}_m)$ $(m >n)$, we always assume that $v$ is embedded into $U(\mathfrak{t}_m)$ by the canonical embedding $\iota\colon U(\mathfrak{t}_n)\rightarrow  U(\mathfrak{t}_m)$ given by $\iota(t_{ij})=t_{ij}$. Similarly to $\A(X)$, the completion of $U(\mathfrak{t}_n)$ with respect to degree is denoted by $\widehat{U}(\mathfrak{t}_n)$.

\subsection{Lie algebra weight systems for Jacobi diagrams}
\label{subsec:Lie_alg_weight}

Let $(\g, \langle \cdot, \cdot \rangle)$ be a metrized Lie algebra over $\K$, that is, a Lie algebra $\g$ over $\K$ equipped with a symmetric non-degenerate  bilinear form $\langle \cdot, \cdot \rangle \colon  \g\otimes \g \to \K$ which is ad-invariant in the sense of \cite[Definition~A.1]{CDM12}, i.e., $\langle \ad_z(x), y \rangle + \langle x, \ad_z(y)\rangle =0 $ for any $x,y,z \in \g$, where $\ad_z(x)=[z,x]$ for any $z,x \in \g$. 
In this paper, when it is clear from the context, we often denote a metrized Lie algebra $(\g, \langle \cdot, \cdot \rangle)$ by $\g$ for brevity.
There is the canonical element $\mathbf{1}_{\g} \in \End_{\K}(\g)= \g \otimes \g^{\ast}$ corresponding to the identity map. 
Note that there is an isomorphism $\g^{\ast} \cong \g$ induced by $\langle \cdot, \cdot \rangle$. 
Therefore, the element $\mathbf{1}_{\g}$ can be regarded as an element of $\g \otimes \g$ and is called the \emph{quadratic Casimir tensor} with respect to $\langle \cdot, \cdot \rangle$. 
It is expressed as $\mathbf{1}_{\g} = \sum_{a=1}^{\dim \g} \bm{v}_a \otimes \bm{v}_a^{\ast}$ for any basis $\{\bm{v}_a\}$ of $\g$ and the dual basis $\{\bm{v}_a^{\ast}\}$ with respect to $\langle \cdot, \cdot \rangle$, i.e., $\langle \bm{v}_a, \bm{v}_b^{\ast}\rangle = \delta_{ab}$. 
In particular, if we take an orthonormal basis $\{\bm{e}_a\}$ of $\g$ with respect to $\langle \cdot, \cdot \rangle$, we have the expression $\mathbf{1}_{\g} = \sum_{a=1}^{\dim \g} \bm{e}_a \otimes \bm{e}_a$. The Lie bracket $[\cdot, \cdot] \in \g^{\ast}\otimes \g^{\ast} \otimes \g$ can also be regarded as an element $\mathbf{t}_{\g} \in \g^{\otimes 3}$ and expressed as $\mathbf{t}_{\g}  = \sum_{i,j,k} c_{ijk} \bm{e}_i \otimes \bm{e}_j \otimes \bm{e}_k$
 in terms of the structure constant $c_{ijk}$ defined as $[\bm{e}_i,\bm{e}_j] = \sum_k c_{ijk} \bm{e}_k$. As is well-known, the constants $c_{ijk}$ are totally antisymmetric tensor and so $\mathbf{t}_{\g}  \in \bigwedge^3 \g$ (\cite[Lemma~A.2, Corollary~A.3]{CDM12}).
 
\begin{example}
We give two typical examples which will be used in Section~\ref{sec:weight_system}.
    \begin{enumerate}[label=(\arabic*)]
        \item For $\g=\gl_N(\K)$, we can take the trace form $B_0$ as $\langle \cdot, \cdot \rangle$, where $B_0$ is defined as  $B_0(x,y) = \tr(xy)$ for $x,y \in \g$.
        \item For a semi-simple Lie algebra $\g$, we can take the Killing form $B$ as $\langle \cdot, \cdot \rangle$  and it is defined by $B(x,y) = \tr(\ad_x \ad_y)$ for $x,y \in \g$. It is known that the Killing form $B$ is proportional to the trace form $B_0$ (cf.~\cite[Section~A.2]{CDM12}).
    \end{enumerate}
\end{example}

Next, we recall the universal $\g$ weight systems for $\A(X)$ in the case where $X$ is the disjoint union of $n$ intervals, i.e., $X=\downarrow^{\otimes n}$. For details, see \cite[Proof of Theorem~10]{LeMu96CM} or \cite[Section~6]{CDM12}. 
Let $\A(X)$ be the algebra of Jacobi diagrams based on $X$ as in Section~\ref{subsec:Jacobi}. Suppose that the connected components of $X$ are numbered $1,2, \ldots, n$. Let $\g = (\g, \langle \cdot, \cdot \rangle)$ be a metrized Lie algebra over $\K$ with the universal enveloping algebra $U(\g)$. Then, the algebra homomorphism $W_{\g}\colon \A(X)_\K \rightarrow U(\g)^{\otimes n}$, called the \emph{universal $\g$ weight system} for $\A(X)$, is defined as follows. 
For $i=1,2,\ldots, n$, we suppose that the $J$ has $k_i$ $(\geq 0)$ univalent vertices on $i$th interval denoted by $v_1^i, v_2^i, \ldots, v_{k_i}^i$ with the order following the orientation of the interval. 
We also suppose that $J$ has $l$ $(\geq 0)$ trivalent vertices $v_{e^1_1 e^1_2 e^1_3}, v_{e^2_1 e^2_2 e^2_3}, \ldots, v_{e^l_1 e^l_2 e^l_3}$, where $v_{e^j_1 e^j_2 e^j_3}$ denotes the trivalent vertex incident to the three edges $e^j_1, e^j_2, e^j_3$ with an order representing the cyclic order at the vertex. 
A \emph{state} on the set $E_J$ of the edges of $J$ is a map $\sigma\colon E_J \rightarrow \{1,2,\ldots, \dim \g\}$. 
Any given state $\sigma$ on $E_J$ induces the map $V_J^{\uni} \rightarrow \{1,2,\ldots, \dim \g\}$ on the set $V_J^{\uni}$ of univalent vertices of $J$ by assigning the number $\sigma(e)$ to each univalent vertex $v$ incident to the edge $e$. By abuse of notation, we also denote by $\sigma$ the induced map. Then, we define $W_{\g}(J)\in U(\g)^{\otimes n}$ by
\[
W_{\g}(J) = \sum_{\sigma\colon E_J \rightarrow \{1,2,\ldots, \dim \g\}} \bigotimes _{i=1}^{n} \bm{e}_{\sigma(v_1^i)}\bm{e}_{\sigma(v_2^i)}\cdots  \bm{e}_{\sigma(v_{k_i}^i)} \prod_{j=1}^l \left(-c_{\sigma(e_1^j)\sigma(e_2^j)\sigma(e_3^j)}\right),
\]
where $\{\bm{e}_i\}$ is an orthonormal basis of $\g$ with respect to $\langle \cdot, \cdot \rangle$ and $c_{ijk}$ is the structure constant associated with the basis.
Extending it linearly, we obtain the desired map $W_{\g}\colon \A(X)_\K \rightarrow U(\g)^{\otimes n}$. Note that the well-definedness of the map $W_{\g}$ follows from arguments similar to those in \cite[Section~6.6 and Page~290]{Oht02} or \cite[Section~6.2]{CDM12}.
\begin{remark}
    When defining $W_{\g}$, we assign $-\mathbf{t}_{\g}$ at each trivalent vertex of a Jacobi diagram as in \cite{CDM12}, which differs from \cite[Section~6]{NSS22JT} where instead $\mathbf{t}_{\g}$ is used.
\end{remark}

\begin{example}
    It is known that there is the skein relation to the universal $\sl_2 = \sl_2(\K)$ weight system (\cite[Theorem 6]{CV97}, \cite[Theorem 6.17]{CDM12}) associated with the metrized Lie algebra $(\sl_2, B_0)$, where $B_0$ denotes the trace form on it:
\begin{equation}\label{eq:skein_sl2}
W_{\sl_2} \bigl(\ 
\begin{tikzpicture}[scale=0.4, baseline={(0,-0.07)}, densely dashed]
\draw (0,0) -- (1,0);
\draw (-0.5, 0.7) -- (0,0);
\draw (1,0) -- (1.5, 0.7);
\draw (1,0) -- (1.5, -0.7);
\draw (0,0) -- (-0.5, -0.7);
\end{tikzpicture}
\ \bigr)
= 
2 W_{\sl_2} \bigl(\ 
\begin{tikzpicture}[scale=0.4, baseline={(0,-0.07)}, densely dashed]
\draw[rounded corners=3pt] (-0.5, 0.7) -- (0,0.3) -- (1,0.3) -- (1.5, 0.7);
\draw[rounded corners=3pt] (1.5, -0.7)-- (1,-0.3) -- (0,-0.3) -- (-0.5, -0.7);
\end{tikzpicture}
\ \bigr)
-
2 W_{\sl_2} \bigl(\ 
\begin{tikzpicture}[scale=0.4, baseline={(0,-0.07)}, densely dashed]
\draw (-0.5, 0.7) -- (1.5, -0.7);
\draw (1.5, 0.7)-- (-0.5, -0.7);
\end{tikzpicture}
\ \bigr).
\end{equation}
\end{example}
We next recall the weight systems associated with the representations of a metrized Lie algebra $\g$. We keep the same notation as above. Let $\rho\colon \g \to \End(V)$ be a finite-dimensional representation of $\g$. The tensor product $\rho^{\otimes n}$ extends to an algebra homomorphism $U( \rho)^{\otimes n}\colon U(\g)^{ \otimes n} \to \End(V)^{\otimes n}$. The \emph{weight system associated with the representation} $(\g, \rho)$ is the algebra homomorphism $W_{\g}^{\rho}\colon \A(X)_\K \to \End(V)^{\otimes n}$  defined as the  composition
\[
W_{\g}^{\rho}\colon \A(X)_\K \xrightarrow{W_{\g}} U(\g)^{\otimes n} \xrightarrow{U(\rho)^{\otimes n}} \End(V)^{\otimes n}.
\]
Finally, we briefly describe the general case that $X$ is an oriented compact $1$-manifold. For a Jacobi diagram $J$ based on $X$, we first choose a basepoint on each circle in the connected components of $X$ away from the univalent vertices on it. Then, by applying the same procedure as above and composing the trace maps, we obtain the algebra homomorphism
\[
W_{\g}^{\rho}\colon \A(X)_\K \xrightarrow{W_{\g}} U(\g)^{\otimes b_0(X)} \xrightarrow{U(\rho)^{\otimes b_0(X)}} \End(V)^{\otimes b_0(X)} \xrightarrow{\tr^{\otimes b_1(X)}} \End(V)^{\otimes \chi(X)}.
\]
Here, $b_i(X)$ and $\chi(X)$ denote the $i$th Betti number and the Euler characteristic of $X$ respectively, and the map $\tr^{\otimes b_1(X)}$ is defined as the tensor product of the trace $\tr$ on $\End(V)$ corresponding to each circle in $X$ (and the identity map on $\End(V)$ corresponding to the other connected components).

\subsection{Associators}
\label{subsec:associator}
Let $Y$ be a connected oriented compact $1$-manifold.
We recall from \cite[Section~1]{LeMu96CM} three linear maps 
\[
\varepsilon_Y\colon \A(X\sqcup Y)\to \A(X),\ 
\Delta_Y\colon \A(X\sqcup Y)\to \A(X\sqcup Y\sqcup Y),\ 
S_Y\colon \A(X\sqcup Y)\to \A(X\sqcup \overline{Y}),
\]
where $\overline{Y}$ is $Y$ with opposite orientation.
First, $\varepsilon_i(J)$ is defined to be $J$ if no univalent vertices lie on $Y$, otherwise $0$.
We define $\Delta_Y(J)$ as the sum of all ways of reattaching the legs of $J$ connecting with $Y$ to either $Y$ or its copy.
Finally, define $S_Y(J) = (-1)^k J$, where $k$ is the number of univalent vertices of $J$ lying on $Y$.
When the connected components are indexed by $1,\dots,n$ and $Y$ is the $i$th component, we simply write $\varepsilon_i$, $\Delta_i$, $S_i$, respectively.

Let $R=\exp\Bigl(\frac{1}{2}\ \onechord\ \Bigr) \in \widehat{\A}(\downarrow\downarrow)$ and let $R^{ij} \in \widehat{\A}(\downarrow^{\otimes 3})_\K$ denote the element obtained by inserting $R$ along the $i$th and $j$th strands.
An element $\Phi\in \widehat{\A}(\downarrow\downarrow\downarrow)_\K$ is called an \emph{associator} if it satisfies
\begin{enumerate}[label=(\arabic*)]
    \item $(1\otimes \Phi) \cdot (\Delta_2\Phi) \cdot (\Phi\otimes 1) = (\Delta_3\Phi) \cdot (\Delta_1\Phi)$;
    \item $\Phi^{231} \cdot (\Delta_2R) \cdot \Phi = R^{13} \cdot \Phi^{213} \cdot R^{12}$;
    \item $\varepsilon_i(\Phi)=1$\ \text{for $i=1,2,3$};
    \item $\Phi^{-1} = \Phi^{321}$,
\end{enumerate}
where $\Phi^{ijk}$ denotes the element obtained by arranging the first component to the $i$th component and so on.
See \cite[Section~4]{LeMu96CM} and \cite{Dri91} for details.

Next, we recall horizontal associators, which provide interesting examples of associators and are obtained from Drinfeld assoicators defined as follows. We define algebra homomorphisms $\varepsilon\colon \K\langle\!\langle A, B \rangle\!\rangle \to \K$ and $\Delta \colon \K\langle\!\langle A, B \rangle\!\rangle \to \K\langle\!\langle A, B \rangle\!\rangle \hat{\otimes}  \K\langle\!\langle A, B \rangle\!\rangle$ by $\varepsilon(A)=\varepsilon(B)=0$ and $\Delta(A) = A \otimes 1 + 1 \otimes A, \Delta(B) = B \otimes 1 + 1 \otimes B$, respectively.
A \emph{Drinfeld associator} is a pair $(\mu, \varphi)  \in  \K \times  \K\langle\!\langle A, B \rangle\!\rangle$ satisfying
\begin{enumerate}[label=(\arabic*)]
\item $\mu \neq 0$;
\item $\Delta \varphi = \varphi \otimes \varphi$, $\varepsilon(\varphi)=1$;
\item the \emph{pentagon relation} (\ref{eq:pen}) in $\widehat{U}(\mathfrak{t}_4)_{\K}$ and the two \emph{hexagon relations} (\ref{eq:hex1}), (\ref{eq:hex2}) in $\widehat{U}(\mathfrak{t}_3)_{\K}$:
\end{enumerate}
\begin{gather}
\varphi(t_{12}, t_{23}+t_{24})\varphi(t_{13}+t_{23}, t_{34})  
= \varphi(t_{23}, t_{34})\varphi(t_{12}+t_{13}, t_{24}+t_{34})\varphi(t_{12}, t_{23}), \label{eq:pen} 
\\
\exp\left(\frac{\mu}{2}(t_{13}+t_{23})\right) 
= \varphi(t_{31}, t_{21})\exp\left(\frac{\mu}{2}t_{13}\right)\varphi(t_{13}, t_{32})^{-1}\exp\left(\frac{\mu}{2}t_{23}\right)\varphi(t_{12}, t_{23}), \label{eq:hex1} 
\\
\exp\left(\frac{\mu}{2}(t_{12}+t_{13})\right) 
= \varphi(t_{23}, t_{31})^{-1}\exp\left(\frac{\mu}{2}t_{13}\right) \varphi(t_{21}, t_{13})\exp\left(\frac{\mu}{2}t_{12}\right)\varphi(t_{12}, t_{23})^{-1}. \label{eq:hex2}
\end{gather}
As shown by Furusho~\cite{Fur10}, the pentagon relation \eqref{eq:pen} implies the two hexagon relations \eqref{eq:hex1} and \eqref{eq:hex2} for some $\mu \in \overline{\K}$ which is unique up to sign.
Let $M(\K)$ and $M_1(\K)$ be the set of Drinfeld associators $(\mu, \varphi)$ and the set of Drinfeld associators $(1, \varphi)$ with $\mu=1$, respectively. Note that $M_1(\K) \subset M(\K)$. By definition and simple calculation, one can see that, for a Drinfeld associator $(1, \varphi) \in M_1(\K)$, the element $\Phi=\varphi(t_{12}, t_{23}) \in \widehat{\mathcal{A}}^{\hor}(\downarrow \downarrow \downarrow)_\K$ satisfies the four conditions of associators in the algebra $\widehat{\A}(\downarrow\downarrow\downarrow)_\K$ of Jacobi diagrams. 
We call such an associator obtained from a Drinfeld associator a (\emph{Drinfeld}) \emph{horizontal associator}. 

We recall lower degree terms of two horizontal associators  $\Phi_\KZ$ and $\Phi_\Q^{\even}$ appearing in this paper.
Let $\Phi_\KZ \in \widehat{\mathcal{A}}^{\hor}(\downarrow \downarrow \downarrow)_\C$ denote the \emph{Knizhnik--Zamolodchikov associator}.
Then, $\log \Phi_\KZ$ is given by the image of the series
\begin{align*}
    & \frac{1}{24} [A, B] - \frac{\zeta(3)}{(2\pi i)^3} [ A+ B, [A,B]]  - \frac{1}{1440} \left([A,[A,[A,B]]]] - [B,[B,[B,A]]]]\right)\\
    &\qquad - \frac{1}{11520} \left([A,[B,[A,B]]]] - [B,[A,[B,A]]]]\right) + (\text{deg} \geq 5)
\end{align*}
under the algebra homomorphism $\K\langle\!\langle A, B\rangle\!\rangle \to \widehat{\A}(\downarrow\downarrow\downarrow)_\K$ defined by
\[
A\mapsto \chordA\ ,\quad B\mapsto \chordB\ .
\]
Let $\Phi_\Q^\even \in \widehat{\mathcal{A}}^{\hor}(\downarrow \downarrow \downarrow)$ be an even rational horizontal associator (its existence is proved by Drinfeld in \cite[Theorem~A$''$, Proposition~5.4]{Dri91}).
As shown in Bar-Natan~\cite{Bar97}, lower degree terms of $\log \Phi_\Q^\even$ are obtained from
\begin{align*}
   & \frac{1}{24} [A, B] - \frac{1}{1440} \left([A,[A,[A,B]]]] - [B,[B,[B,A]]]]\right)\\
   &\qquad - \frac{1}{11520} \left([A,[B,[A,B]]]] - [B,[A,[B,A]]]]\right) + (\text{deg} \geq 6)
\end{align*}
via the above algebra homomorphism.
See also \cite[Sections~10.2.5 and 10.4.3]{CDM12} or \cite[pp.~370 and 373]{Oht02}. It is known that $\log \Phi_\Q^\even$ is uniquely determined up to degree $7$ as shown by Bar-Natan~\cite{Bar97}, whereas uniqueness fails at degree $8$ and higher as shown by Brochier~\cite{Broc25}.
According to \cite[Remark~10.38]{CDM12}, $\log \Phi_\KZ = \log \Phi_\Q^\even$ at the degrees $0,1,2,4$.

Let $\Phi$ be an associator and suppose that $F\in \widehat{\A}(\downarrow\downarrow)_\K$
is symmetric (i.e., $F^{21}=F$) and satisfies $\varepsilon_1(F)=\varepsilon_2(F)=1$.
Then, 
\[
(1\otimes F) \cdot (\Delta_2 F) \cdot \Phi \cdot (\Delta_1 F^{-1}) \cdot (F^{-1}\otimes 1)
\]
is again an associator, said to be obtained from $\Phi$ by \emph{twisting} via $F$.
See \cite[Section~7]{LeMu96CM}, \cite[Theorem~A$'$]{Dri91}, \cite[Theorem~10.35]{CDM12}, or \cite[p.~368]{Oht02}.

\begin{example}
\label{ex:twisting}
As written in \cite[Example~10.36]{CDM12}, $\Phi_\KZ$ and $\Phi_\Q^\even$ are related by twisting up to degree $4$.
The twisting by $1 + \alpha J_3$ adds the degree $3$ term $2\alpha [A+B, [A,B]]$ up to degree $4$, where $J_3 \in \A(\downarrow \downarrow)$ is defined as
\[
J_3=
\begin{tikzpicture}[scale=0.4, baseline={(0,0.6)}, densely dashed]
\draw[->,solid] (0,4) -- (0,0);
\draw[->,solid] (2,4) -- (2,0);
\draw (0,1) -- (2,1) ;
\draw (1,1) -- (1,3) ;
\draw (0,3) -- (2,3) ;
\end{tikzpicture}\ ,
\]
and $\alpha \in \K$.
Hence, by taking $\alpha = - \zeta(3)/2(2\pi i)^3$, twisting $\Phi_\Q^\even$ by $1 + \alpha J_3$ gives rise to $\Phi_\KZ$ up to degree $4$. If we take $\alpha = \zeta(3)/(2\pi i)^3$, the corresponding twisting transforms $\Phi_\KZ(A,B)$ to $\Phi_\KZ(-A,-B)$ up to degree $4$.
Here, note that $\Phi_\KZ(A,B) \neq \Phi_\KZ(-A,-B)$  \cite[Remark in page 851]{Dri91}.
By taking any $\alpha \in \Q$, one can obtain another rational (horizontal) associator from $\Phi_\Q^\even$ via twisting up to degree $4$.
\end{example}

\subsection{The Kontsevich invariant}
\label{subsec:Kontsevich}
In \cite{Kon93}, Kontsevich introduced an integral for oriented links $L$ in $S^3$ via the KZ equation.
This gives an invariant of $L$, which takes values in the vector space $\widehat{\A}\left(\bigsqcup^{b_0(L)} S^1\right)_\C$.
Using associators, Le and Murakami~\cite{LeMu96CM} defined the Kontsevich invariant of (framed, oriented) $q$-tangles in a combinatorial way.
See also \cite[Section~10.3.5]{CDM12} and \cite[Section~6.4]{Oht02}.

We briefly review the definition of the Kontsevich invariant of $q$-tangles following \cite[Theorem~2]{LeMu96CM}.
A \emph{$q$-tangle} $T$ is a framed oriented tangle equipped with two non-associative words $w_t(T)$, $w_b(T)$ in letters $+$, $-$, where the letters in $w_t(T)$ (resp.~$w_b(T)$) correspond to the endpoints $\partial_{t}T$ at the top (resp.~$\partial_{b}T$ at the bottom).
For $q$-tangles $T$ and $T'$, let $T\otimes T'$ denote the $q$-tangle obtained by horizontal juxtaposition of $T$ and $T'$.
When $w_b(T)=w_t(T')$, the composition $T\circ T'$ is defined by stacking $T$ on $T'$.
Let us define the \emph{Kontsevich invariant} $Z_\Phi$ associated with an associator $\Phi$.
We first define $Z_\Phi(T)$ for elementary $q$-tangles $T$:
\begin{gather*}
\begin{tikzpicture}[scale=0.4, baseline={(0,0.3)}]
    \draw[->] (2,2) -- (0,0) node[at start, anchor=south]{\small $+$} node[at end, anchor=north]{\small $+$};
    \draw (0,2) -- (0.8,1.2) node[at start, anchor=south]{\small $+$};
    \draw[->] (1.2,0.8) -- (2,0) node[at end, anchor=north]{\small $+$};
    \node at (3.5,1) {$\mapsto$};
    \node at (7.5,1) {$\exp\bigl(\frac{1}{2}\ \onechordcross\ \bigr)$};
\end{tikzpicture},
\quad
\begin{tikzpicture}[scale=0.4, baseline={(0,0.3)}]
    \draw[->] (0,2) -- (2,0) node[at start, anchor=south]{\small $+$} node[at end, anchor=north]{\small $+$};
    \draw (2,2) -- (1.2,1.2) node[at start, anchor=south]{\small $+$};
    \draw[->] (0.8,0.8) -- (0,0) node[at end, anchor=north]{\small $+$};
    \node at (3.5,1) {$\mapsto$};
    \node at (7.5,1) {$\exp\bigl(\frac{-1}{2}\ \onechordcross\ \bigr)$};
\end{tikzpicture},
\\
\begin{tikzpicture}[scale=0.4, baseline={(0,0.3)}]
    \draw[->] (0,2) -- (0,0) node[at start, anchor=south]{\small $+\phantom{)}$} node[at end, anchor=north]{\small $(+$};
    \draw[->] (1,2) -- (1,0) node[at start, anchor=south]{\small $(+$} node[at end, anchor=north]{\small $+)$};
    \draw[->] (2,2) -- (2,0) node[at start, anchor=south]{\small $+)$} node[at end, anchor=north]{\small $\phantom{(}+$};
\end{tikzpicture}
\mapsto \Phi,
\quad
\begin{tikzpicture}[scale=0.4, baseline={(0,0.3)}]
    \draw[->] (0,2) -- (0,0) node[at start, anchor=south]{\small $(+$} node[at end, anchor=north]{\small $+\phantom{)}$};
    \draw[->] (1,2) -- (1,0) node[at start, anchor=south]{\small $+)$} node[at end, anchor=north]{\small $(+$};
    \draw[->] (2,2) -- (2,0) node[at start, anchor=south]{\small $\phantom{(}+$} node[at end, anchor=north]{\small $+)$};
\end{tikzpicture}
\mapsto \Phi^{-1},
\\
\begin{tikzpicture}[scale=0.4, baseline={(0,0.0)}]
    \draw[->] (1,0) arc[start angle=0, end angle=180, radius=1] node[at start, anchor=north]{\small $-$} node[at end, anchor=north]{\small $+$};
    \node at (3,0) {$\mapsto$}; 
    \node[draw, circle, inner sep=1] at (5,1) {$\nu$};
    \draw (6,0) arc[start angle=0, end angle=65, radius=1];
    \draw[<-] (4,0) arc[start angle=180, end angle=115, radius=1];
\end{tikzpicture}\ ,
\quad
\begin{tikzpicture}[scale=0.4, baseline={(0,0.0)}]
    \draw[->] (-1,0) arc[start angle=180, end angle=360, radius=1] node[at start, anchor=south]{\small $+$} node[at end, anchor=south]{\small $-$};
    \node at (3,0) {$\mapsto$}; 
    \draw[->] (4,0) arc[start angle=180, end angle=360, radius=1];
\end{tikzpicture}\ .
\end{gather*}
Here, $\nu\in \widehat{\A}(\downarrow)$ is defined by
\[
\nu=
\Biggl(\ 
\begin{tikzpicture}[baseline={(0,0.3)}]
\draw[->, rounded corners=5.7pt] (0, 1.2) -- (0, 0.5) -- (0, 0.1) -- (0.3, -0.3) -- (0.6, 0.1) -- (0.6, 0.7) -- ( 0.9, 1.1)-- (1.2, 0.7) -- ( 1.2, -0.4);
\node[draw,  rectangle, fill=white, minimum width=1.6cm,minimum height=0.6cm] at (0.6,0.4) {$S_2 \Phi$};
\end{tikzpicture}
\ \Biggr)^{-1},
\]
which is known to be independent of the choice of $\Phi$.
Since every $q$-tangle decomposes into elementary $q$-tangles by $\otimes$ and $\circ$, we can define $Z_\Phi(T)$ so that it satisfies $Z_\Phi(T\otimes T') = Z_\Phi(T)\otimes Z_\Phi(T')$ and $Z_\Phi(T\circ T') = Z_\Phi(T)\cdot Z_\Phi(T')$.
Note that the normalization in \cite{LeMu96CM} is the same as in \cite{CDM12} but different from \cite{Oht02}.

\begin{example}\label{ex:3-pb_konts_inv}
    Let $T$ be the $q$-braid of three strands given as follows:
    \[
    T = \begin{tikzpicture}[scale=0.4, baseline={(0,0.3)}]
    \draw[->] (0,2) -- (0,0);
    \draw (1,2) -- (1.35,1.65);
    \draw[->,rounded corners=3pt] (1.65,1.35) -- (2,1)-- (1,0) ;
    \draw[rounded corners=3pt] (2,2) -- (1,1) -- (1.35,0.65);
    \draw[->] (1.65, 0.35) -- (2,0);
    \foreach \x in {2.75, -0.75}{
    \node at (-0.1,\x) {$(+$};
    \node at (1.1,\x) {$+)$};
    \node at (2,\x) {$+$};
    }
\end{tikzpicture}.
    \]
    Then, we observe directly that $Z_{\Phi_\KZ}(T)$ have terms whose coefficients are multiples of $\zeta(3)$ as 
    \begin{align*}
     \log Z_{\Phi_\KZ}(T) &= \log \left(\Phi_\KZ(t_{12}, t_{23})^{-1} \cdot \exp(t_{23}) \cdot \Phi_\KZ(t_{12}, t_{23})\right) \\
     &= t_{23} - \frac{1}{24 }[[t_{12},t_{23}],t_{23}] \\
    &\quad + \frac{\zeta(3)}{(2 \pi i)^3} [[ t_{12} + t_{23}, [t_{12}, t_{23}]], t_{23}] + (\text{deg} \geq 5),
    \end{align*}
    and so we conclude that $Z_{\Phi_\KZ}(T) \neq Z_{\Phi_\Q^\even}(T)$.
\end{example}

We here recall from \cite[Theorem~7]{LeMu96CM} the behavior of $Z_\Phi$ under twisting $\Phi$, which plays a key role in the proof of Theorem~\ref{thm:deg_8}.
Let $\Phi$, $\Phi'$ be associators and suppose that $\Phi$ is obtained by twisting $\Phi'$ via $F\in \widehat{\A}(\downarrow\downarrow)_\K$.
For a $2$-component string link $T$, we have 
\[
F\cdot Z_\Phi(T)\cdot F^{-1} = Z_{\Phi'}(T).
\]
See also \cite[Theorem~10.37]{CDM12}.
In \cite[Theorem~3.14]{Fur20}, a similar formula is given for proalgebraic string link as in  Section~\ref{subsec:GT_act_string}.

\section{The action of the Grothendieck--Teichm\"uller group on proalgebraic string links}

In this section, we start by recalling the notion of proalgebraic string links and the action of the Grothendieck--Teichm\"{u}ller group on them, and then give the proof of Theorem~\ref{thm:GT_1_2-string_link} under the assumption of Theorem~\ref{thm:deg_8}.

\subsection{Proalgebraic string links}\label{subsec:string_link}
Here, we set up the notation related to proalgebraic string links. For precise definitions, see \cite{Fur20}.

Let $n$ be a positive integer and let $\epsilon=(\epsilon_1, \epsilon_2, \ldots, \epsilon_n)$ be a sequence of letters $+$ and $-$. For such a sequence $\epsilon$, we denote by $\SL_\epsilon $ the monoid of string links of type $\epsilon$, that is, oriented string links $T$ with trivial framings such that the endpoints $\partial_t(T)$ and $\partial_b(T)$ correspond to the sequence $\epsilon$. In the present article, a string link $T$ of type $\epsilon$ is regarded as a $q$-tangle with the fixed non-associative word $w_t(T) = w_b(T) = ((\cdots ((( \epsilon_1 \epsilon_2) \epsilon_3 )\epsilon_4)\cdots) \epsilon_n)$ so that the Kontsevich invariant for $q$-string links is compatible with the morphism $\rho_{\epsilon}(p)$ for string links defined in \cite[Proposition~3.25(3)]{Fur20} as explained in Remark~\ref{rem:3.2}. 
The group of invertible elements of $\SL_\epsilon$ is equal to the group of pure braids of type $\epsilon$ and denoted by $\PB_\epsilon$. 
Since most of the results we consider follow from the case that $\epsilon=(++\cdots +)$ with  small modifications, we mainly consider $\SL_n= \SL_{(++\cdots +)}$ for simplicity. 
The pure braid group $\PB_n \subset \SL_n $ is generated by 
\[
x_{ij} = 
\begin{tikzpicture}[baseline={(0,0.5)}]
\foreach \x in {0, 0.5, 1.5, 4.5, 5.5, 6}
\draw[->,rounded corners=4pt] (\x,1.4) -- (\x, 0);
\draw (2, 1.4) .. controls (2, 0.8) and (2.6, 0.9) .. (3.9, 0.8);
\draw[->] (4.1, 0.8) .. controls (5.4, 0.7) and (2, 0.6) .. (2, 0);
\draw (2.5, 1.4) -- (2.5, 0.95);
\draw (2.5, 0.8) -- (2.5, 0.5);
\draw[->] (2.5, 0.3) -- (2.5, 0);
\draw (3.5, 1.4) -- (3.5, 0.9);
\draw (3.5, 0.75) -- (3.5, 0.65);
\draw[->] (3.5, 0.5) -- (3.5, 0);
\draw (4, 1.4) -- (4, 0.75);
\draw[->] (4, 0.55) -- (4, 0);
\foreach \x / \y in {0/1, 0.5/, 1.5/, 2/i, 2.5/, 3.5/, 4/j, 4.5/, 5.5/, 6/n}
\node at (\x, -0.25) {$\y$};
\foreach \x in {1, 3, 5}
\node at (\x, 0.7
) {$\cdots$};
\end{tikzpicture}\ ,
\]
where $1 \leq i < j \leq n$. For the relations they satisfy, see \cite{KT08}.
It is well known that there is a semidirect product decomposition $\PB_3 \cong F_2(x_{12}, x_{23}) \rtimes \langle x_{13} \rangle$, where $F_2(x_{12}, x_{23}) $ denotes the free group of rank $2$ generated by $x_{12}$ and $x_{23}$, and $\langle x_{13} \rangle \cong \Z$ is the cyclic group generated by $x_{13}$. For $1 \leq a \leq a + \alpha < b \leq b + \beta \leq n$, we set
\begin{align*}
x_{a\cdots a + \alpha, b \cdots b +\beta} &\coloneqq (x_{a, b} x_{a, b+1} \cdots x_{a, b+\beta})\cdot (x_{a+1, b} x_{a+1, b+1} \cdots x_{a+1, b+\beta})\\
&\qquad \cdots (x_{a+\alpha, b} x_{a+\alpha, b+1} \cdots x_{a+\alpha, b+\beta}) \in \PB_n,
\end{align*}
which can also be obtained from $x_{a,b} \in \PB_{n-\alpha - \beta}$ by doubling the $a$th and $b$th strings $\alpha$ and $\beta$ times, respectively (cf.~\cite[Section~1.1]{Fur20}). 

Let $\K[\SL_\epsilon]$ denote the $\K$-algebra generated by string links of type $\epsilon$. The topological $\K$-algebra of \emph{proalgebraic string links} is defined by its completion $\widehat{\K[\SL_\epsilon]}$ by a singular filtration \`a la Vassiliev. 
More precisely, a singular string link is considered as the element of $\K[\SL_\epsilon]$ by resolving its double points as follows:
\[
\begin{tikzpicture}
\draw[->] (0,1) -- (1, 0);
\draw[->] (1,1) -- (0, 0);
\node at (0.5, 0.5) {\textbullet}; 
\node at (1.5, 0.5) {$=$};
\draw (2,1) -- (2+0.4, 1-0.4);
\draw[->]  (2+0.6, 1-0.6) -- (3, 0);
\draw[->] (3,1) -- (2, 0);
\node at (3.5, 0.5) {$-$};
\draw (5,1) -- (5-0.4, 1-0.4);
\draw[->] (5-0.6,1-0.6) -- (4, 0);
\draw[->] (4,1) -- (5, 0);
\end{tikzpicture}\ .
\]
For each $k \geq 0$, define $J_k \subset \K[\SL_\epsilon]$ to be the subspace generated by singular string links of type $\epsilon$ with $k$-double points. 
We then have the filtration $\K[\SL_\epsilon] = J_0 \supset J_1 \supset J_2 \supset \cdots \supset J_k \supset \cdots$, called the {\it Vassiliev filtration} on $\K[\SL_\epsilon]$. 
We now define $\widehat{\K[\SL_\epsilon]}$ as the inverse limit $\varprojlim_k \K[\SL_\epsilon]/J_k$. 
It is known that the Vassiliev filtration $\{J_k\}_k$ on $\K[\PB_{\epsilon}]$ coincides with the filtration by the powers $I^k$ of the augmentation ideal $I \coloneqq \Ker(\K[\PB_{\epsilon}] \to \K)$ (cf.~\cite[Section~12.2]{CDM12}).

\subsection{Grothendieck--Teichm\"uller group action on proalgebraic string links}\label{subsec:GT_act_string}
Here, we recall from \cite{Fur20} the action of the proalgebraic Grothendieck--Teichm\"uller group on the proalgebraic string links inspired by the work of Kassel--Turaev~\cite{KaTu98}. 
In this paper, we review only the facts necessary for our purposes.
Hence, we refer the interested reader to \cite{Dri91}, \cite{KaTu98}, and \cite{Fur20} for further details.

Let $F_2(\K) \subset \widehat{\K[F_2]}$ be the Malcev completion of the free group $F_2 = F_2(x,y)$ generated by $x$ and $y$. Here, we consider that $\widehat{\K[F_2]}$ is endowed with the Hopf algebra structure induced by that of $\K[F_2]$ with group-like elements $g\in F_2$.
Let $\GT(\K)$ denote the \emph{proalgebraic Grothendieck--Teichm\"uller group}, i.e., we set
$$
\GT(\K)=\{ (\lambda, f) \in \K^{\times} \times F_2(\K)\  |\  (\lambda, f)\  \text{satisfies the following (\ref{eq:gt1})--(\ref{eq:gt3})}\},
$$
\begin{align}
&f(x, y)=f(y, x)^{-1},\label{eq:gt1}\\
& f(z, x)z^m f(y, z)y^m f(x, y)x^m = 1 \quad \text{if $xyz=1$ and $m=\frac{\lambda -1 }{2}$}, \label{eq:gt2}\\
& f(x_{12}, x_{23}x_{24})f(x_{13}x_{23}, x_{34})=f(x_{23}, x_{34})f(x_{12}x_{13}, x_{24}x_{34})f(x_{12}, x_{23}) \label{eq:gt3},
\end{align}
where \eqref{eq:gt3} is an equality in $\widehat{\K[\PB_4]}$.
Here, $a^m$ is defined by $a^m = \exp(m \log a)$ for $a \in \{x, y, z\}$ and note that it is well-defined in $\widehat{\K[F_2]}$. The group structure of $\GT(\K)$ is given as
\[
(\lambda_1, f_1) \bullet (\lambda_2, f_2) = (\lambda_1 \lambda_2, f_2 \cdot (f_1(x^{\lambda_2}, f_2^{-1} y^{\lambda_2} f_2))).
\]
Note that, as shown by Furusho~\cite{Fur10}, the pentagon relation \eqref{eq:gt3} implies the hexagon relations \eqref{eq:gt1} and \eqref{eq:gt2} for some $\lambda \in \overline{\K}$ which is unique up to sign.
Let $\GT_1(\K)$ denote its unipotent part, i.e., the subgroup of $\GT(\K)$ consisting of elements with $\lambda=1$.

Next, recall the graded variant of $\GT(\K)$, called the \emph{graded Grothendieck--Teichm\"uller group} and denoted by $\GRT(\K)$.
As a set, $\GRT(\K)$ consists of pairs $(c, g)\in \K^{\times} \times \K\langle\!\langle A, B \rangle\!\rangle$ such that 
\begin{enumerate}[label=(\arabic*)]
\item $\Delta(g)=g\otimes g$, $\varepsilon(g)=1$;
\item $g$ satisfies the pentagon relation (\ref{eq:pen}), and the two hexagon relations (\ref{eq:hex1}) and (\ref{eq:hex2}) with $\mu=0$.
\end{enumerate}
For any $(c_1, g_1), (c_2, g_2) \in \GRT(\K)$, the product $(c_1, g_1) \bullet  (c_2, g_2)$ is defined as
\[
(c_1, g_1) \bullet  (c_2, g_2) = \left(c_1 c_2, g_2 \cdot g_1\left(\frac{X}{c_2}, g_2^{-1} \frac{Y}{c_2} g_2 \right)\right).
\]
Note that, as a consequence of the result by Furusho~\cite{Fur10}, the two hexagon relations \eqref{eq:hex1} and \eqref{eq:hex2} are redundant since these relations follow from the pentagon relation \eqref{eq:pen}. 
Then, the surjective homomorphism $\GRT(\K) \to \K^{\times}$ sending $(c,g)$ to $c$ induces the short exact sequence
\[
0 \rightarrow \GRT_1(\K) \rightarrow \GRT(\K) \rightarrow \K^{\times} \rightarrow 0,
\]
and, by choosing a section $\K^{\times} \to \GRT(\K)$ given as $c \mapsto (c,1)$, one obtains the semidirect product decomposition $\GRT(\K) \cong \K^{\times} \ltimes \GRT_1(\K)$. Here, $\GRT_1(\K)$ denotes the subgroup of $\GRT(\K)$ consisting of elements with $c=1$. 
As explained in \cite[Section~2]{Fur06}, the action of $\K^{\times}$ given by $(1, g(A, B))\mapsto (1, g(A/c, B/c))$ endows (the ring of regular functions on) $\GRT_1(\K)$ a grading structure.

As shown in \cite{Dri91}, the set $M(\K)$ of associators forms a $(\GRT(\K), \GT(\K))$-bitorsor, i.e., $\GRT(\K)$ acts on $M(\K)$ freely and transitively from the left and also $\GT(\K)$ acts on $M(\K)$ freely and transitively from the right, and the actions are commutative each other\footnote{Following \cite{Fur20}, we reverse the order of the product given in \cite{Dri91}.}. 
These actions are given as follows:
For $(c, g) \in \GRT(\K)$ and $(\mu, \varphi) \in M(\K)$, 
\[
(c, g) \bullet (\mu, \varphi) = \left(\frac{\mu}{c}, g\cdot \varphi\left(\frac{A}{c}, g^{-1} \frac{B}{c} g \right) \right) \in M(\K)
\]
and, for $(\mu, \varphi) \in M(\K)$ and $(\lambda, f) \in \GT(\K)$, 
\[
(\mu, \varphi)\bullet (\lambda, f)=(\lambda \mu, \varphi\cdot f( e^{\mu A}, \varphi^{-1} e^{\mu B} \varphi) )\in M(\K).
\]

An extension of the $\GT(\K)$-action on the proalgebraic (pure) braid groups by Drinfeld (\cite{Dri91}) to the proalgebraic tangle is indicated by Kassel--Turaev (\cite{KaTu98}) and rigorously constructed by Furusho (\cite{Fur20}). Here, we recall this extension in the case of proalgebraic string links. We have a $\GT(\K)$-action
\[
\rho_\epsilon \colon \GT(\K) \to \mathrm{Aut}(\widehat{\K[\mathcal{SL}_\epsilon]})
\]
and a $\GRT(\K)$-action
\[
\rho_\epsilon \colon \GRT(\K) \rightarrow \mathrm{Aut}(\widehat{\A}(\downarrow^{\epsilon})_{\K}),
\]
both of which restrict to the actions on $\widehat{\K[\PB_n]}$ and $\widehat{U}(\mathfrak{t}_n)_{\K}$ given in \cite{Dri91}, respectively. Moreover, there is a morphism of bitorsors 
\[
\rho_{\epsilon}\colon M(\K) \to \Isom(\widehat{\K[\mathcal{SL}_\epsilon]}, \widehat{\A}(\downarrow^{\epsilon})_{\K}).
\]
The above actions are given by combination of the operator $S$, the A-part, B-part, and C-part of the ABC-construction defined in \cite[Section~2]{Fur20}.
See Figures~\ref{fig:gt_action} and \ref{fig:grt_action}.

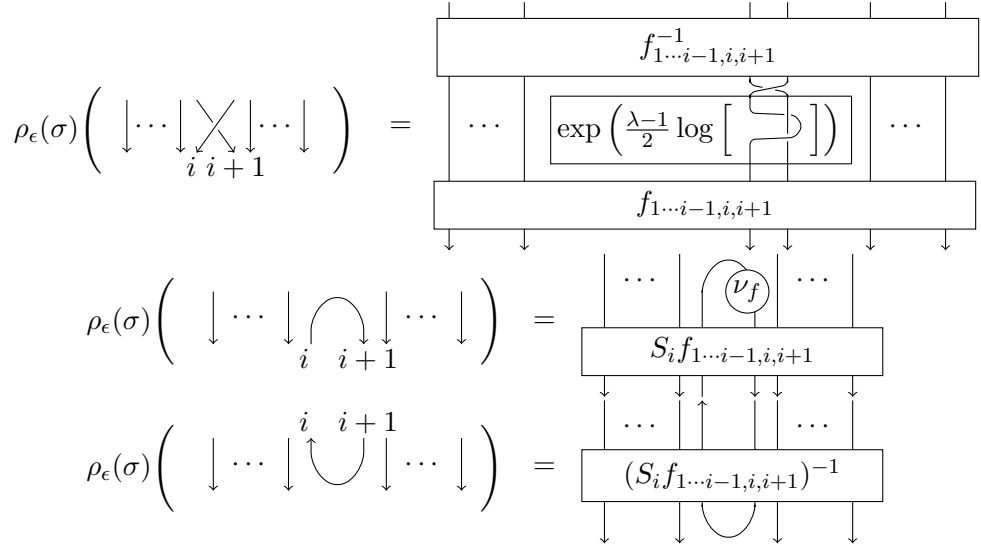
\begin{figure}[h]
\begin{tikzpicture}
\begin{scope}[xshift=0cm]
\begin{scope}[xshift=-1cm]
\node at (-4.7, 0.3) {$\rho_\epsilon(\sigma)\Biggl($};
\node at (-0.9, 0.3) {$\Biggr)$};
\begin{scope}[xscale=0.714, xshift=-1.0cm]
\draw[->,rounded corners=4pt] (-2.3, 0.7) -- (-3, 0);  
\draw[->,rounded corners=4pt] (-2, 0.7) -- (-2, 0);  
\node at (-1.5, 0.35) {$\cdots$};
\draw[->,rounded corners=4pt] (-3.3, 0.7) -- (-3.3, 0);  
\node at (-3.8, 0.35) {$\cdots$};
\draw[->, rounded corners=4pt] (-4.3, 0.7) -- (-4.3, 0);  
\draw[->, rounded corners=4pt] (-1, 0.7) -- (-1, 0);  
\node at (-2.45, -0.2) {$i \  i+1$};
\draw[rounded corners=4pt] (-3, 0.7) -- (-2.7, 0.4);  
\draw[->, rounded corners=4pt] (-2.6, 0.3) -- (-2.3, 0);  
\node at (0.8, 0.35) {$=$};
\end{scope}

\end{scope}
\begin{scope}[xshift=3.5cm]
\draw[<-, rounded corners=2pt] (0, -1.3) -- (0, 0.175) -- (0.6, 0.2) -- (0.7, 0.35) -- (0.6, 0.5) -- (0, 0.525) -- (0, 0.7) -- (0, 0.8) -- (0.5, 0.9) -- (0.5, 1) -- (0.5, 2);
\draw[color=white, line width=2.5pt] (0.5, 0.7) -- (0.5, 0.25);
\draw[rounded corners=2pt] (0, 2) -- (0, 1) -- (0, 0.9) -- (0.166, 0.9-0.033);
\draw[rounded corners=2pt] (0.5-0.166, 0.8+0.033) -- (0.5, 0.8) -- (0.5, 0.7) -- (0.5, 0.25) ;
\draw[->, rounded corners=2pt]  (0.5, 0.15) -- (0.5, -1.3);
\node at (-1.4, 0.3) {$\exp \Big(\frac{\lambda -1}{2} \log\Big[
$};
\node at (1,0.3) {$\Big]\Big)$};
 \node[draw,  rectangle, minimum width=4cm,minimum height=0.9cm] at (-0.65, 0.3) {};

\begin{scope}[xshift=0.3cm]
\draw[->, rounded corners=4pt] (1.3, 2) -- (1.3, -1.3);  
\node at (1.8, 0.35) {$\cdots$};
\draw[->, rounded corners=4pt] (2.3, 2) -- (2.3, -1.3);  
\end{scope}
\begin{scope}[xshift=-1.2cm]
\draw[->, rounded corners=4pt] (-1.8, 2) -- (-1.8, -1.3);  
\node at (-2.3, 0.35) {$\cdots$};
\draw[->, rounded corners=4pt] (-2.8, 2) -- (-2.8, -1.3);  
\end{scope}

 \node[draw,  rectangle, fill=white, minimum width=7.2cm,minimum height=0.6cm] at (-0.6,-0.7) {$f_{1\cdots i-1, i, i+1}$};
  \node[draw,  rectangle, fill=white, minimum width=7.2cm,minimum height=0.6cm] at (-0.56,1.4) {$f_{1\cdots i-1, i, i+1}^{-1}$};
 \end{scope}
\end{scope}
\end{tikzpicture}
\begin{tikzpicture}
\node at (-1.8, 0.6) {$\rho_\epsilon(\sigma)\Biggl($};
\begin{scope}[yshift=0.3cm]
\draw[->, rounded corners=6pt]  (0.6,0) -- (0.6,0.4) -- (0.95, 0.7) -- (1.3, 0.4) --  (1.3, 0); 
\draw[->, rounded corners=6pt]  (1.6,0.7) --(1.6, 0); 
\node at (1.1, -0.2) {$i\quad i+1$};
\node at (2.1, 0.35) {$\cdots$};
\node at (-0.2, 0.35) {$\cdots$};
\draw[->, rounded corners=6pt]  (2.6,0.7) --(2.6, 0); 
\draw[->, rounded corners=6pt]  (0.3,0.7) --(0.3, 0); 
\draw[->, rounded corners=6pt]  (-0.7,0.7) --(-0.7, 0); 
\end{scope}
\node at (3, 0.6) {$\Biggr)$};
\node at (3.7, 0.6) {$=$};
\begin{scope}[xshift=5.2cm, yshift=0.8cm]
\draw[->, rounded corners=6pt] (0.6, -1.2) -- (0.6,0) -- (0.6, 0.4) -- (0.95, 0.7) -- (1.3, 0.4) --  (1.3, 0) -- (1.3, -1.2);
\node[draw,  circle, inner sep=1, fill=white] at (1.2,0.2) {$\nu_f$};
\draw[->, rounded corners=6pt]  (1.6,0.7) --(1.6, -1.2); 
\node at (2.1, 0.35) {$\cdots$};
\node at (-0.2, 0.35) {$\cdots$};
\draw[->, rounded corners=6pt]  (2.6,0.7) --(2.6, -1.2); 
\draw[->, rounded corners=6pt]  (0.3,0.7) --(0.3, -1.2); 
\draw[->, rounded corners=6pt]  (-0.7,0.7) --(-0.7, -1.2); 
 \node[draw,  rectangle, fill=white, minimum width=4cm,minimum height=0.6cm] at (1,-0.6) {$S_if_{1\cdots i-1, i, i+1}$};
\end{scope}
\end{tikzpicture}

\begin{tikzpicture}
\node at (-1.8, 0.6) {$\rho_\epsilon(\sigma)\Biggl($};
\begin{scope}[yshift=0.3cm]
\draw[->, rounded corners=6pt]  (1.3, 0.7) -- (1.3, 0.3) -- (0.95, 0) -- (0.6, 0.3) -- (0.6, 0.7);
\draw[->, rounded corners=6pt]  (1.6,0.7) --(1.6, 0); 
\node at (2.1, 0.35) {$\cdots$};
\node at (-0.2, 0.35) {$\cdots$};
\node at (1.1, 0.9) {$i\quad i+1$};
\draw[->, rounded corners=6pt]  (2.6,0.7) --(2.6, 0); 
\draw[->, rounded corners=6pt]  (0.3,0.7) --(0.3, 0); 
\draw[->, rounded corners=6pt]  (-0.7,0.7) --(-0.7, 0); 
\end{scope}
\node at (3, 0.6) {$\Biggr)$};
\node at (3.7, 0.6) {$=$};
\begin{scope}[xshift=5.2cm, yshift=0.8cm]
\begin{scope}[yshift=-1.2cm]
\draw[->, rounded corners=6pt] (1.3, 1.9) -- (1.3, 0.7) -- (1.3, 0.3) -- (0.95, 0) -- (0.6, 0.3) -- (0.6, 0.7) -- (0.6, 1.9);
\end{scope}
\draw[->, rounded corners=6pt]  (1.6,0.7) --(1.6, -1.2); 
\node at (2.1, 0.35) {$\cdots$};
\node at (-0.2, 0.35) {$\cdots$};
\draw[->, rounded corners=6pt]  (2.6,0.7) --(2.6, -1.2); 
\draw[->, rounded corners=6pt]  (0.3,0.7) --(0.3, -1.2); 
\draw[->, rounded corners=6pt]  (-0.7,0.7) --(-0.7, -1.2); 
 \node[draw,  rectangle, fill=white, minimum width=4cm,minimum height=0.6cm] at (1,-0.3) {$(S_if_{1\cdots i-1, i, i+1})^{-1}$};
\end{scope}
\end{tikzpicture}
\caption{The $\GT(\K)$-action on elementary tangles. Here, $\sigma=(\lambda, f) \in \GT(\K)$ and we set $f_{1\cdots i-1, i, i+1} \coloneqq f(x_{1 \cdots i-1, i}, x_{i i+1}) = f(x_{1i} x_{2i} \cdots x_{i-1i}, x_{ii+1})$ when $i \geq 2$ and $f_{1\cdots i-1, i, i+1}=1$ when $i=1$. Note that $\exp(\frac{\lambda-1}{2} \log (x_{i, i+1}))$ is well-defined in $\widehat{\K[\mathit{\PB}_n]}$. 
The element $\nu_f \in \widehat{\K[\mathcal{SL}_1]}$ is defined in the same way as $\nu$ in Section~\ref{subsec:Kontsevich}, by replacing $\Phi$ with $f$.}
\label{fig:gt_action}
\end{figure}

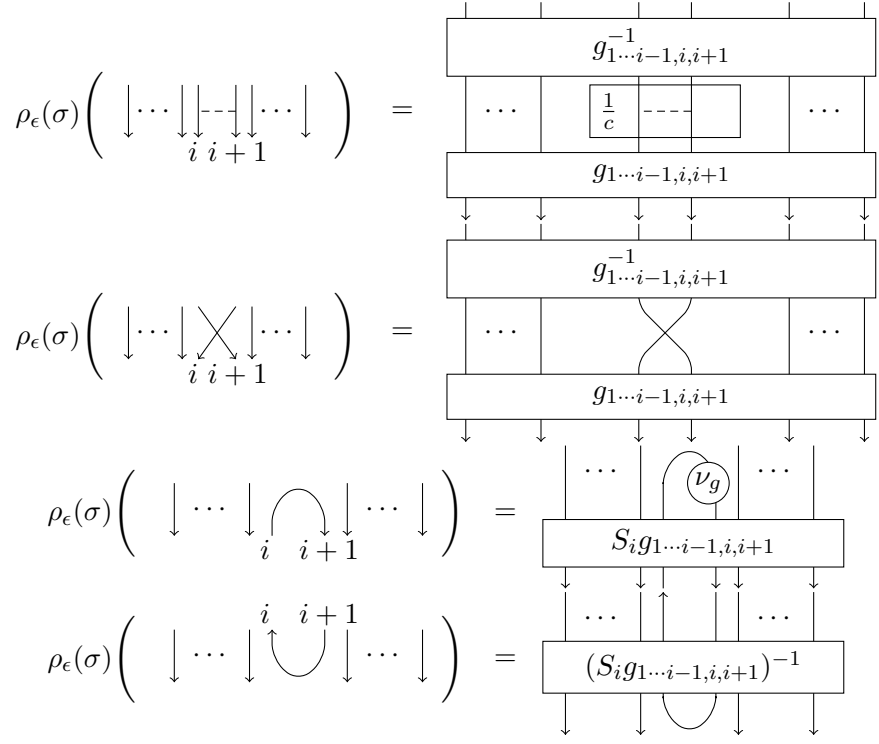
\begin{figure}[h]
\begin{tikzpicture}
\begin{scope}[xshift=1.7cm]
\node at (-4.7, 0.3) {$\rho_\epsilon(\sigma)\Biggl($};
\node at (-0.9, 0.3) {$\Biggr)$};
\begin{scope}[xscale=0.714, xshift=-1.0cm]
\draw[->, rounded corners=4pt] (-2.3, 0.7) -- (-2.3, 0);  
\draw[->, rounded corners=4pt] (-2, 0.7) -- (-2, 0);  
\node at (-1.5, 0.35) {$\cdots$};
\draw[->, rounded corners=4pt] (-3.3, 0.7) -- (-3.3, 0);  
\node at (-3.8, 0.35) {$\cdots$};
\draw[->, rounded corners=4pt] (-4.3, 0.7) -- (-4.3, 0);  
\draw[->, rounded corners=4pt] (-1, 0.7) -- (-1, 0);  
\node at (-2.45, -0.2) {$i \  i+1$};
\draw[->, rounded corners=4pt] (-3, 0.7) -- (-3, 0);
\draw[densely dashed, rounded corners=4pt] (-2.3, 0.35) -- (-3, 0.35);
\node at (0.8, 0.35) {$=$};
\end{scope}

\begin{scope}[xshift=6cm]
\draw[->, rounded corners=4pt] (-2.3, 1.8) -- (-2.3, -1.1);  
\draw[->, rounded corners=4pt] (-2+1, 1.8) -- (-2+1, -1.1);  
\node at (-1.5+1, 0.35) {$\cdots$};
\draw[->, rounded corners=4pt] (-3.3-1, 1.8) -- (-3.3-1, -1.1);  
\node at (-3.8-1, 0.35) {$\cdots$};
\draw[->, rounded corners=4pt] (-4.3-1, 1.8) -- (-4.3-1, -1.1);  
\draw[->, rounded corners=4pt] (-1+1, 1.8) -- (-1+1, -1.1);  
\node at (-3.4, 0.35) {$\frac{1}{c}$};
 \node[draw,  rectangle, minimum width=2cm,minimum height=0.7cm] at (-2.65, 0.35) {};
\draw[->, rounded corners=4pt] (-3, 1.8) -- (-3, -1.1);
\draw[densely dashed, rounded corners=4pt] (-2.3, 0.35) -- (-3, 0.35);
 \node[draw,  rectangle, fill=white, minimum width=5.7cm,minimum height=0.6cm] at (-2.7,-0.5) {$g_{1\cdots i-1, i, i+1}$};
  \node[draw,  rectangle, fill=white, minimum width=5.7cm,minimum height=0.6cm] at (-2.7,1.2) {$g_{1\cdots i-1, i, i+1}^{-1}$};
\end{scope}
\end{scope}
\end{tikzpicture}

\begin{tikzpicture}
\begin{scope}[xshift=1.7cm]
\node at (-4.7, 0.3) {$\rho_\epsilon(\sigma)\Biggl($};
\node at (-0.9, 0.3) {$\Biggr)$};
\begin{scope}[xscale=0.714, xshift=-1.0cm]
\draw[->, rounded corners=4pt] (-2.3, 0.7) -- (-3, 0);  
\draw[->, rounded corners=4pt] (-3, 0.7) -- (-2.3, 0);  
\node at (-2.45, -0.2) {$i \  i+1$};
\draw[->, rounded corners=4pt] (-2, 0.7) -- (-2, 0);  
\node at (-1.5, 0.35) {$\cdots$};
\draw[->, rounded corners=4pt] (-3.3, 0.7) -- (-3.3, 0);  
\node at (-3.8, 0.35) {$\cdots$};
\draw[->, rounded corners=4pt] (-4.3, 0.7) -- (-4.3, 0);  
\draw[->, rounded corners=4pt] (-1, 0.7) -- (-1, 0);  
\node at (0.8, 0.35) {$=$};
\end{scope}
\begin{scope}[xshift=6cm]
\draw[->, rounded corners=4pt] (-2.3, 1.8) -- (-2.3, 0.7) -- (-3, 0) -- (-3, -1.1);  
\draw[->, rounded corners=4pt] (-3, 1.8) --  (-3, 0.7) -- (-2.3, 0) -- (-2.3, -1.1);  
\draw[->, rounded corners=4pt] (-2+1, 1.8) -- (-2+1, -1.1);  
\node at (-1.5+1, 0.35) {$\cdots$};
\draw[->, rounded corners=4pt] (-3.3-1, 1.8) -- (-3.3-1, -1.1);  
\node at (-3.8-1, 0.35) {$\cdots$};
\draw[->, rounded corners=4pt] (-4.3-1, 1.8) -- (-4.3-1, -1.1);  
\draw[->, rounded corners=4pt] (-1+1, 1.8) -- (-1+1, -1.1);  
 \node[draw,  rectangle, fill=white, minimum width=5.7cm,minimum height=0.6cm] at (-2.7,-0.5) {$g_{1\cdots i-1, i, i+1}$};
  \node[draw,  rectangle, fill=white, minimum width=5.7cm,minimum height=0.6cm] at (-2.7,1.2) {$g_{1\cdots i-1, i, i+1}^{-1}$};
\end{scope}
\end{scope}
\end{tikzpicture}

\begin{tikzpicture}
\node at (-1.8, 0.6) {$\rho_\epsilon(\sigma)\Biggl($};
\begin{scope}[yshift=0.3cm]
\draw[->, rounded corners=6pt]  (0.6,0) -- (0.6,0.4) -- (0.95, 0.7) -- (1.3, 0.4) --  (1.3, 0); 
\draw[->, rounded corners=6pt]  (1.6,0.7) --(1.6, 0); 
\node at (1.1, -0.2) {$i\quad i+1$};
\node at (2.1, 0.35) {$\cdots$};
\node at (-0.2, 0.35) {$\cdots$};
\draw[->, rounded corners=6pt]  (2.6,0.7) --(2.6, 0); 
\draw[->, rounded corners=6pt]  (0.3,0.7) --(0.3, 0); 
\draw[->, rounded corners=6pt]  (-0.7,0.7) --(-0.7, 0); 
\end{scope}
\node at (3, 0.6) {$\Biggr)$};
\node at (3.7, 0.6) {$=$};
\begin{scope}[xshift=5.2cm, yshift=0.8cm]
\draw[->, rounded corners=6pt] (0.6, -1.2) -- (0.6,0) -- (0.6, 0.4) -- (0.95, 0.7) -- (1.3, 0.4) --  (1.3, 0) -- (1.3, -1.2);
\node[draw,  circle, inner sep=1, fill=white] at (1.2,0.2) {$\nu_g$};
\draw[->, rounded corners=6pt]  (1.6,0.7) --(1.6, -1.2); 
\node at (2.1, 0.35) {$\cdots$};
\node at (-0.2, 0.35) {$\cdots$};
\draw[->, rounded corners=6pt]  (2.6,0.7) --(2.6, -1.2); 
\draw[->, rounded corners=6pt]  (0.3,0.7) --(0.3, -1.2); 
\draw[->, rounded corners=6pt]  (-0.7,0.7) --(-0.7, -1.2); 
 \node[draw,  rectangle, fill=white, minimum width=4cm,minimum height=0.6cm] at (1,-0.6) {$S_ig_{1\cdots i-1, i, i+1}$};
\end{scope}
\end{tikzpicture}

\begin{tikzpicture}
\node at (-1.8, 0.6) {$\rho_\epsilon(\sigma)\Biggl($};
\begin{scope}[yshift=0.3cm]
\draw[->, rounded corners=6pt]  (1.3, 0.7) -- (1.3, 0.3) -- (0.95, 0) -- (0.6, 0.3) -- (0.6, 0.7);
\draw[->, rounded corners=6pt]  (1.6,0.7) --(1.6, 0); 
\node at (2.1, 0.35) {$\cdots$};
\node at (-0.2, 0.35) {$\cdots$};
\node at (1.1, 0.9) {$i\quad i+1$};
\draw[->, rounded corners=6pt]  (2.6,0.7) --(2.6, 0); 
\draw[->, rounded corners=6pt]  (0.3,0.7) --(0.3, 0); 
\draw[->, rounded corners=6pt]  (-0.7,0.7) --(-0.7, 0); 
\end{scope}
\node at (3, 0.6) {$\Biggr)$};
\node at (3.7, 0.6) {$=$};
\begin{scope}[xshift=5.2cm, yshift=0.8cm]
\begin{scope}[yshift=-1.2cm]
\draw[->, rounded corners=6pt] (1.3, 1.9) -- (1.3, 0.7) -- (1.3, 0.3) -- (0.95, 0) -- (0.6, 0.3) -- (0.6, 0.7) -- (0.6, 1.9);
\end{scope}
\draw[->, rounded corners=6pt]  (1.6,0.7) --(1.6, -1.2); 
\node at (2.1, 0.35) {$\cdots$};
\node at (-0.2, 0.35) {$\cdots$};
\draw[->, rounded corners=6pt]  (2.6,0.7) --(2.6, -1.2); 
\draw[->, rounded corners=6pt]  (0.3,0.7) --(0.3, -1.2); 
\draw[->, rounded corners=6pt]  (-0.7,0.7) --(-0.7, -1.2); 
 \node[draw,  rectangle, fill=white, minimum width=4cm,minimum height=0.6cm] at (1,-0.3) {$(S_ig_{1\cdots i-1, i, i+1})^{-1}$};
\end{scope}
\end{tikzpicture}
\caption{The $\GRT(\K)$-action on fundamental infinitesimal tangles. Here, $\sigma=(c,g) \in \GRT(\K)$ and we set $g_{1\cdots i-1, i, i+1} = \Delta_1^{i-2} g$ when $i \geq 2$ and $g_{1\cdots i-1, i, i+1}=1$ when $i=1$. 
The element $\nu_g \in \widehat{\mathcal{A}}(\downarrow)_{\K}$ is defined in the same way as $\nu$ in Section~\ref{subsec:Kontsevich}, by replacing $\Phi$ with $g$.}
\label{fig:grt_action}
\end{figure}

\begin{remark}\label{rem:3.1}
We note that, in \cite{KaTu98} and \cite{Fur20}, the above $\GT(\K)$-action is given, more generally, for proalgebraic tangles. 
\end{remark}
\begin{remark}\label{rem:3.2}
For $p=(1, \varphi) \in M_1(\K)$, the isomorphism $\rho_{\epsilon}(p)$ coincides with the Kontsevich invariant $Z_{\Phi}$ for $\Phi=\varphi(t_{12},t_{23})$ as briefly explained in the following. 
We have an isomorpshim $\widehat{\K[\SL_\epsilon]} \cong \varprojlim_n \K[\SL^\pre_\epsilon]/\T^\pre_n$ \cite[Proposition~3.7(2)]{Fur20}.
Let $\K[\I\SL^\pre_\epsilon]$ denote the $\K$-vector space generated by ``infinitesimal pre-string links'' of type $\epsilon$ subject to the relations \cite[(IT1)--(IT6)]{Fur20} and its completion $\widehat{\K[\I\SL_\epsilon]}$ is the $\K$-algebra of ``infinitesimal string links'' of type $\epsilon$.
For $p = (\mu, \varphi) \in M(\K)$, an isomorphism $\rho_{\epsilon}(p)\colon \widehat{\K[\SL_{\epsilon}]} \to \widehat{\K[\I\SL_{\epsilon}]}$ is defined in \cite[Proposition~3.25(3)]{Fur20}.
There is a natural isomorphism $\widehat{\K[\I\SL_{\epsilon}]}\to \widehat{\A}(\downarrow^\epsilon)_\K$ defined in \cite[Proposition~3.28]{Fur20}.
When $p = (1, \varphi)$ with $\Phi= \varphi(t_{12},t_{23})$, their composition coincides with the induced isomorphism $Z_\Phi\colon \widehat{\K[\SL_{\epsilon}]}\to \widehat{\A}(\downarrow^\epsilon)_\K$.
See \cite[Notation~2.7(2), Proposition~2.19, and Remark~3.29]{Fur20}.
\end{remark}

In \cite{Fur20}, Furusho studied the case of string links in more detail and gave the following theorem.

\begin{theorem}[{\cite[Theorem~A]{Fur20}}]\label{thm:fur20_thm}
Let  $\GT_1(\K)$ denote the unipotent part of $\GT(\K)$. Then, the following holds.
\begin{enumerate}[label=\textup{(\arabic*)}]
\item The $\GT_1(\K)$-action on $\widehat{\K[\mathcal{SL}_\epsilon]}$ is given by an inner conjugation.
    \item The $\GT_1(\K)$-action on  proalgebraic knots is trivial.
\end{enumerate}
\end{theorem}

As a consequence of Theorem~\ref{thm:fur20_thm}(2), the $\GT(\K)$-action on proalgebraic $n$-component string links is not faithful for $n=1$, which is in contrast to the faithfulness for $n\geq 3$. A way to see an inner conjugate action of $\GT_1(\K)$ in Theorem~\ref{thm:fur20_thm}(1) is as follows \cite[Theorem~4.14]{Fur20}: By the $(\GRT(\K), \GT(\K))$-bitorsor structure of $M(\K)$, there is an isomorphism $r_p \colon \GT_1(\K) \overset{\sim}{\to} \GRT_1(\K)$ for $p=(1,\varphi) \in M_1(\K)$. For $\sigma=(1,f) \in \GT_1(\K)$, we have the corresponding element $r_p(\sigma) \in GRT_1(\K)$ via $p \bullet \sigma = r_p(\sigma) \bullet p $ for $ p=(1,\varphi) \in M_1(\K)$. Then, for $\Gamma \in \widehat{\K[\mathcal{SL}_{\epsilon}}]$, we have
\begin{align*}
Z_{\Phi}(\rho_{\epsilon}(\sigma)(\Gamma)) &= Z_{\Phi\bullet \sigma}(\Gamma) = Z_{r_p(\sigma)\bullet \Phi}(\Gamma)=\rho_{\epsilon} (r_p(\sigma))(Z_{\Phi}(\Gamma))\\
&= 
F(r_p(\sigma), \epsilon)\cdot Z_{\Phi}(\Gamma)\cdot F(r_p(\sigma), \epsilon)^{-1}.
\end{align*}
Here, $F(r_p(\sigma), \epsilon)$ is an element of $\widehat{\A}(\downarrow^{\epsilon})_{\K}$ defined by using a twistor $F(r_p(\sigma))\in \widehat{\A}(\downarrow \downarrow)_\K$ that is obtained as $\GRT_1(\K)$-analogue of twistor lemma~\cite[Lemma~4.2]{Fur20}. 
Applying $Z_{\Phi}^{-1}$ to both sides, we obtain the  conjugation action as desired
\begin{equation}\label{eq:conj_act}
\rho_\epsilon(\sigma)(\Gamma)=\varpi(\sigma, \epsilon)\cdot \Gamma\cdot \varpi(\sigma, \epsilon)^{-1},
\end{equation}
where we set 
$\varpi(\sigma, \epsilon)\coloneqq Z_{\Phi}^{-1}(F(r_p(\sigma),\epsilon)) \in \widehat{\K[\mathcal{SL}_{\epsilon}}]^{\times}$.

Since it is useful to clarify the relationship between a twisting of associators recalled in Section~\ref{subsec:Kontsevich} and a twistor of an element in $\GRT_1(\K)$, we prepare the following lemma.
\begin{lemma}\label{lem:twistor}
    For any $\sigma =(1,g) \in \GRT_1(\K)$ and $p=(1, \varphi) \in M_1(\K)$, we set $\Phi'= \varphi'(t_{12},t_{23})$ and $\Phi=\varphi(t_{12}, t_{23})$, where $p' = (1, \varphi') = \sigma \bullet (1, \varphi) \in M_1(\K)$. Then, the twisting of $\Phi$ via a twistor $F(\sigma) \in \widehat{\A}(\downarrow \downarrow)_{\K}$ results in the associator $\Phi'$.
\end{lemma}

\begin{proof}
    It is implicit in \cite{Fur20}. We set $\epsilon_3 = +{+}+$. By definition of a twistor $F(\sigma)$ given in \cite[Lemma~4.2]{Fur20}, we have $g=F(\sigma,\epsilon_3)^{321} \cdot F(\sigma,\epsilon_3)^{-1}$, where we put
   \begin{align*}
   F(\sigma, \epsilon_3) &= (F(\sigma) \otimes 1)\cdot \Delta_1(F(\sigma)), \\
   F(\sigma, \epsilon_3)^{321} &= (1 \otimes F(\sigma)) \cdot \Delta_2(F(\sigma)).
   \end{align*}
   As shown in \cite[Proposition~4.7]{Fur20}, we have $g^{-1} t_{23} g= F(\sigma, \epsilon_3)\cdot t_{23}  \cdot F(\sigma, \epsilon_3)^{-1}$ and $t_{12}=F(\sigma, \epsilon_3) \cdot t_{12} \cdot F(\sigma, \epsilon_3)^{-1} \in \A(\downarrow\downarrow\downarrow)_\K$ since $F(\sigma,\epsilon_3)^{321} \cdot t_{23} = t_{23} \cdot F(\sigma,\epsilon_3)^{321}$ and $t_{12} \cdot F(\sigma,\epsilon_3) = F(\sigma,\epsilon_3) \cdot t_{12}$ hold. Therefore, by definition of $\GRT_1(\K)$-action on $M_1(\K)$, we have
   \begin{align*}
   \Phi' &= g\cdot \Phi(t_{12}, g^{-1}t_{23}g)\\
   &= (F(\sigma,\epsilon_3)^{321} \cdot F(\sigma,\epsilon_3)^{-1}) \cdot (F(\sigma, \epsilon_3) \cdot \Phi(t_{12}, t_{23}) \cdot F(\sigma, \epsilon_3)^{-1})\\
   &=F(\sigma, \epsilon_3)^{321}\cdot \Phi(t_{12}, t_{23}) \cdot F(\sigma, \epsilon_3)^{-1}
   \end{align*}
   and so we conclude that the twisting of $\Phi$ via $F(\sigma)$ yields $\Phi'$ as desired.
\end{proof}

\subsection{The action of the Grothendieck--Teichm\"uller group on $2$-component proalgebraic string links}
\label{subsec:GT-action}

In this subsection, assuming Theorem~\ref{thm:deg_8}, we give the proof of Theorem~\ref{thm:GT_1_2-string_link}.
\begin{proof}[Proof of Theorem~\ref{thm:GT_1_2-string_link}]
We first give the proof for the sequence $\epsilon = {++}$. 
By \eqref{eq:conj_act}, it suffices to show that 
\[
F(r_p(\sigma))\cdot  Z_{\Phi}(\Gamma)\cdot  F(r_p(\sigma))^{-1} \neq Z_{\Phi}(\Gamma)
\]
for some $\sigma \in \GT_1(\K)$, $\Gamma \in \widehat{\K[\mathcal{SL}_{\epsilon}}]$, $p=(1,\varphi)\in M_1(\K)$, and $\Phi=\varphi(t_{12},t_{23})$. Here, $F(r_p(\sigma)) \in \widehat{\A}(\downarrow \downarrow)_\K$ is a twistor of $r_p(\sigma)\in \GRT_1(\K)$. By the construction of $F(r_p(\sigma))$ and Lemma~\ref{lem:twistor}, it is equivalent to show
\begin{equation}\label{eq:ineq_z_phi}
  Z_{\Phi'}(\Gamma) \neq Z_{\Phi}(\Gamma),
\end{equation}
where $\Phi'=\varphi'(t_{12}, t_{23})$ is given by $p'=(1, \varphi')= r_p(\sigma) \bullet p = p \bullet \sigma$ for some $\sigma \in \GT_1(\K)$, $\Gamma \in \widehat{\K[\mathcal{SL}_{\epsilon}}]$, and $p=(1,\varphi)\in M_1(\K)$. By  Theorem~\ref{thm:deg_8}, we can take a pair of horizontal associators $\Phi, \Phi'$ and a $2$-component string link $T \in \SL_2$ which satisfy \eqref{eq:ineq_z_phi}. 
Therefore, by choosing $\sigma \in \GT_1(\K)$ which maps $\Phi$ to $\Phi'$ and $\Gamma = \iota(T)$, where $\iota\colon \SL_2 \to \widehat{\K[\mathcal{SL}_2]}$ is the induced map by the canonical inclusion $\iota\colon \SL_2 \to \K[\mathcal{SL}_2]$, we conclude that $\sigma \in \GT_1(\K)$ acts non-trivially on $\Gamma = \iota(T) \in \widehat{\K[\mathcal{SL}_2]}$.
For other sequences $\epsilon$, the proofs are almost identical and so we omit the detail. This completes the proof.
\end{proof}

\section{The algebra $\A(\downarrow\downarrow)$ and the proof of Theorem~\ref{thm:deg_8}}
\label{sec:main_theorem}

In this section, we precisely study the algebra $\A(\downarrow\downarrow)$ of Jacobi diagrams based on two intervals and prove Theorem~\ref{thm:deg_8}.

\subsection{The algebra structure of $\A(\downarrow\downarrow)$}
\label{subsec:2_arrows}

The proof of the next lemma is based on the idea kindly explained to us by Katsumi Ishikawa.

\newcommand{\symmetric}[1]{
\begin{tikzpicture}[scale=0.4, baseline={(0,0.3)}, densely dashed]
    \draw[solid,->] (0,2) -- (0,0);
    \draw[solid,->] (3,2) -- (3,0);
    \draw (0,1) -- (1,1);
    \draw (3,1) -- (2,1);
    \fill[fill=gray!50] (1.5,1) circle [radius=0.7];
    \node at (1.5,1) {$#1$};
\end{tikzpicture}
}

\begin{lemma}
\label{lem:box}
The sum $\symmetric{J}\ +\ \symmetric{\rotatebox{180}{$J$}}$ is contained in the center of $\A(\downarrow\downarrow)$, where the second term is obtained from the first term by rotating $J$ $180$ degrees.
In particular, if these two are the same, then $\symmetric{J}$ is in the center. 
\end{lemma}

\begin{proof}
In this proof, we use a ``box'' to express a sum of diagrams:
\[
\begin{tikzpicture}[scale=0.4, baseline={(0,0.5)}, densely dashed]
    \draw[solid] (0,3) -- (0,2);
    \draw[solid] (0,1) -- (0,0);
    \draw[solid] (2,3) -- (2,2);
    \draw[solid] (2,1) -- (2,0);
    \draw[solid] (-1, 1) -- (3, 1) -- (3, 2) -- (-1, 2) -- cycle;
    \draw (3,1.5) -- (4,1.5);
    \draw[dotted] (4,1.5) -- (5,1.5);
\end{tikzpicture}
\ =\ 
\begin{tikzpicture}[scale=0.4, baseline={(0,0.5)}, densely dashed]
    \draw[solid] (0,3) -- (0,0);
    \draw[solid] (2,3) -- (2,0);
    \draw (0,1.5) -- (4,1.5);
    \draw[dotted] (4,1.5) -- (5,1.5);
\end{tikzpicture}
\ +\ 
\begin{tikzpicture}[scale=0.4, baseline={(0,0.5)}, densely dashed]
    \draw[solid] (0,3) -- (0,0);
    \draw[solid] (2,3) -- (2,0);
    \draw (2,1.5) -- (4,1.5);
    \draw[dotted] (4,1.5) -- (5,1.5);
\end{tikzpicture}
\ .
\]
It follows from the IHX relation that a box commutes with any Jacobi diagram.
In particular, 
\[
\begin{tikzpicture}[scale=0.4, baseline={(0,0.9)}, densely dashed]
    \draw[solid] (0,5) -- (0,4);
    \draw[solid] (0,3) -- (0,2);
    \draw[solid,->] (0,1) -- (0,0);
    \draw[solid] (2,5) -- (2,4);
    \draw[solid] (2,3) -- (2,2);
    \draw[solid,->] (2,1) -- (2,0);
    \draw[solid] (-1, 3) -- (3, 3) -- (3, 4) -- (-1, 4) -- cycle;
    \draw[solid] (-1, 1) -- (3, 1) -- (3, 2) -- (-1, 2) -- cycle;
    \draw (3,3.5) .. controls +(1,0) and +(0,0.1) .. (4,3.2);
    \draw (3,1.5) .. controls +(1,0) and +(0,-0.1) .. (4,1.8);
    \fill[fill=gray!50] (4,2.5) circle [radius=0.7];
    \node at (4,2.5) {$J$};
\end{tikzpicture}
\ =\ 
\begin{tikzpicture}[scale=0.4, baseline={(0,0.9)}, densely dashed]
    \draw[solid,->] (0,5) -- (0,0);
    \draw[solid,->] (2,5) -- (2,0);
    \draw (0,3.5) .. controls +(1,0) and +(0,0.1) .. (1,3.2);
    \draw (0,1.5) .. controls +(1,0) and +(0,-0.1) .. (1,1.8);
    \fill[fill=gray!50] (1,2.5) circle [radius=0.7];
    \node at (1,2.5) {$J$};
\end{tikzpicture}
\ +\ 
\begin{tikzpicture}[scale=0.4, baseline={(0,0.9)}, densely dashed]
    \draw[solid,->] (0,5) -- (0,0);
    \draw[solid,->] (2,5) -- (2,0);
    \draw (2,3.5) .. controls +(1,0) and +(0,0.1) .. (3,3.2);
    \draw (2,1.5) .. controls +(1,0) and +(0,-0.1) .. (3,1.8);
    \fill[fill=gray!50] (3,2.5) circle [radius=0.7];
    \node at (3,2.5) {$J$};
\end{tikzpicture}
\ +\ 
\begin{tikzpicture}[scale=0.4, baseline={(0,0.9)}, densely dashed]
    \draw[solid,->] (0,5) -- (0,0);
    \draw[solid,->] (3,5) -- (3,0);
    \draw (0,2.5) -- (1,2.5);
    \draw (3,2.5) -- (2,2.5);
    \fill[fill=gray!50] (1.5,2.5) circle [radius=0.7];
    \node at (1.5,2.5) {$J$};
\end{tikzpicture}
\ +\ 
\begin{tikzpicture}[scale=0.4, baseline={(0,0.9)}, densely dashed]
    \draw[solid,->] (0,5) -- (0,0);
    \draw[solid,->] (3,5) -- (3,0);
    \draw (0,2.5) -- (1,2.5);
    \draw (3,2.5) -- (2,2.5);
    \fill[fill=gray!50] (1.5,2.5) circle [radius=0.7];
    \node at (1.5,2.5) {\rotatebox{180}{$J$}};
\end{tikzpicture}
\]
lies in the center.
Here, the first and second terms are obviously in the center, and thus $\symmetric{J}\ +\ \symmetric{\rotatebox{180}{$J$}}$ is in the center.
\end{proof}

\begin{corollary}[{\cite[Exercise in Section~5.11.2]{CDM12}}] \label{cor:A_leq2_is_in_center}
$\A_{\leq 2}(\downarrow\downarrow)$ is contained in the center of $\A(\downarrow\downarrow)$.
\end{corollary}

\begin{proof}
By the AS and STU relations, it suffices to show that $\onechord$ and $\phichord$ are in the center, which follows from Lemma~\ref{lem:box}.
\end{proof}

For a finite set $C$, let $\A(C)$ denote the $\Q$-vector space generated by Jacobi diagrams whose univalent vertices are colored by elements of $C$ subject to the AS and IHX relations.
For $n\geq 3$, define $J_n \in \A(\downarrow\downarrow)$ and $\widetilde{J}_n \in \A(\{a,b\})$ by
\[
J_n=
\begin{tikzpicture}[scale=0.4, baseline={(0,0.6)}, densely dashed]
\draw[->,solid] (0,4.5) -- (0,-0.5);
\draw[->,solid] (2,4.5) -- (2,-0.5);
\draw (1,4) -- (1,0);
\draw (0,4) -- (2,4);
\draw (0,3) -- (1,3);
\node at (0.5,2.3){$\vdots$};
\draw (0,1) -- (1,1);
\draw (0,0) -- (2,0);
\end{tikzpicture}
\quad\text{and}\quad
\widetilde{J}_n=
\begin{tikzpicture}[scale=0.4, baseline={(0,0.6)}, densely dashed]
\draw (1,4) -- (1,0);
\draw (0,4) -- (2,4) node[at start, anchor=east]{$a$} node[at end, anchor=west]{$b$};
\draw (0,3) -- (1,3) node[at start, anchor=east]{$a$};
\node at (0.5,2.3){$\vdots$};
\draw (0,1) -- (1,1) node[at start, anchor=east]{$a$};
\draw (0,0) -- (2,0) node[at start, anchor=east]{$a$} node[at end, anchor=west]{$b$};
\end{tikzpicture}\ ,
\]
where each Jacobi diagram has $n+1$ univalent vertices.

\begin{proposition}
\label{prop:J3}
Every element in $\A_{\leq 3}(\downarrow\downarrow)$ commutes each other.
\end{proposition}

\begin{proof}
Let $J$ be a Jacobi diagram in $\A_{3}(\downarrow\downarrow)$.
When (the dashed part of) $J$ is not connected, by the STU relation, $J$ is expressed as a sum of connected Jacobi diagrams of degree $3$ and products of diagrams of $\deg\leq 2$.
Thus, by Corollary~\ref{cor:A_leq2_is_in_center}, it suffices to show that every connected diagram of degree $3$ commutes each other.
Recall here that the AS, IHX, and STU relations give useful identities 
\[
 \begin{tikzpicture}[scale=0.4, baseline={(0,0.4)}, densely dashed]
 \draw (-2,0) -- (2,0);
 \draw (-1,0) -- (0,1);
 \draw (1,0) -- (0,1);
 \draw (0,3) -- (0,1);
\end{tikzpicture}
=
\frac{1}{2}
\begin{tikzpicture}[scale=0.4, baseline={(0,0.4)}, densely dashed]
 \draw (-2,0) -- (2,0);
 \draw (0,3) -- (0,2);
 \draw (0,0) -- (0,1);
 \draw (0,1.5) circle [radius=0.5];
\end{tikzpicture}\ ,
\qquad
\begin{tikzpicture}[scale=0.4, baseline={(0,0.4)}, densely dashed]
 \draw[solid,->] (-2,0) -- (2,0);
 \draw (-1,0) -- (0,1);
 \draw (1,0) -- (0,1);
 \draw (0,3) -- (0,1);
\end{tikzpicture}
=
\frac{1}{2}
\begin{tikzpicture}[scale=0.4, baseline={(0,0.4)}, densely dashed]
 \draw[solid,->] (-2,0) -- (2,0);
 \draw (0,3) -- (0,2);
 \draw (0,0) -- (0,1);
 \draw (0,1.5) circle [radius=0.5];
\end{tikzpicture}\ .
\]
Using these relations, we show that any connected diagram of degree $3$ can be expressed as a linear combination of  
$J_3$ and Jacobi diagrams obtained from
\[
\begin{tikzpicture}[scale=0.4, baseline={(0,0.4)}, densely dashed]
    \draw[solid,->] (0,3) -- (0,0);
    \draw[solid,->] (2,3) -- (2,0);
    \draw (0,1) .. controls +(1,0) and +(1,0) .. (0,2);
\end{tikzpicture}
\ ,\qquad 
\begin{tikzpicture}[scale=0.4, baseline={(0,0.4)}, densely dashed]
    \draw[solid,->] (0,3) -- (0,0);
    \draw[solid,->] (2,3) -- (2,0);
    \draw (2,1) .. controls +(1,0) and +(1,0) .. (2,2);
\end{tikzpicture}
\ ,\qquad 
\begin{tikzpicture}[scale=0.4, baseline={(0,0.4)}, densely dashed]
    \draw[solid,->] (0,3) -- (0,0);
    \draw[solid,->] (2,3) -- (2,0);
    \draw (0,1.5) -- (2,1.5);
\end{tikzpicture}
\]
by inserting a bubble
(\ $
\begin{tikzpicture}[scale=0.4, baseline={(0,0.3)}, densely dashed]
    \draw (0,1) -- (3,1);
\end{tikzpicture}
\ \leadsto\ 
\begin{tikzpicture}[scale=0.4, baseline={(0,0.3)}, densely dashed]
    \draw (0,1) -- (1,1);
    \draw (3,1) -- (2,1);
    \draw (2,1) arc[start angle=0, end angle=360, radius=0.5];
\end{tikzpicture}
$\ ) twice.
Therefore, Corollary~\ref{cor:A_leq2_is_in_center} completes the proof.
\end{proof}

As a corollary of Proposition~\ref{prop:J3}, we show Proposition~\ref{prop:deg_6}.

\begin{proof}[Proof of Proposition~\ref{prop:deg_6}]
There exists $F\in \widehat{\A}(\downarrow\downarrow)$ such that $\Phi'$ is obtained from $\Phi$ by twisting via $F$. 
It follows from \cite[Theorem~7]{LeMu96CM} that $F\cdot Z_{\Phi}(T)\cdot F^{-1} = Z_{\Phi'}(T)$, and hence it suffices to show $(F\cdot Z_{\Phi}(T)\cdot F^{-1})_{\leq 6} = Z_{\Phi}(T)_{\leq 6}$.
Let $z_d$ and $F_d$ be the degree $d$ terms of $Z_{\Phi}(T)$ and $F$, respectively.
Then, Corollary~\ref{cor:A_leq2_is_in_center} implies
\begin{align*}
F\cdot Z_{\Phi}(T)\cdot F^{-1} 
&\equiv 1+z_1+z_2+F\cdot z_3\cdot F^{-1}+z_4+z_5+z_6 \\
&\equiv 1+z_1+z_2+z_3+F_3\cdot z_3-z_3\cdot F_3+z_4+z_5+z_6 
\end{align*}
modulo $\deg\geq 7$.
Since we have $F_3\cdot z_3=z_3\cdot F_3$ by Proposition~\ref{prop:J3}, it completes the proof.
\end{proof}

\subsection{Proof of Theorem~\ref{thm:deg_8} and Proposition~\ref{prop:general_replacement}}
\label{subsec:main_theorem}

The goal of this subsection is to prove Theorem~\ref{thm:deg_8}, and the main task is to establish the following key proposition.

\begin{proposition}
\label{prop:J3J2n+1}
$J_3$ does not commute with the Jacobi diagram $J_{2n+1}$ for $n\geq 2$.
\end{proposition}

Let $\A$ be $\A(X)$ or $\A(C)$, where $X$ is an oriented compact $1$-manifold.
We write $\A^c$ for the the subspace of $\A$ generated by connected diagrams and $\A^t$ for the quotient of $\A$ by declaring $J=0$ if the first Betti number of (the dashed part of) $J$ is positive.
We also define $\A^{c,t}$ to be the image of $\A^{c}$ under the projection $\A\to \A^{t}$.
Note that $\A^t(C)$ is regarded as a subspace of $\A(C)$.

Let $\downarrow^C$ denote the disjoint union of intervals indexed by elements of $C$.
Define a linear map $\tau\colon \A^{c}(\downarrow^C)\to \A^{c,t}(C)$ as follows: $\tau(J)=0$ if $b_1(J)>0$, otherwise $\tau(J)$ is the diagram obtained from $J$ by eliminating the intervals and by assigning $c_i\in C$ to a univalent vertex if it is attached to the interval labeled by $c_i$.
By definition, it induces $\tau\colon \A^{c,t}(\downarrow^C)\to \A^{c,t}(C)$.
Furthermore, the Poincar\'e--Birkhoff--Witt isomorphism $\chi\colon \A(C)\to \A(\downarrow^C)$ induces $\A^{c,t}(C)\to \A^{c,t}(\downarrow^C)$, which is the inverse of $\tau$ (cf.~\cite[Section~5.7.1]{CDM12}).

For $J,J'\in \A^{c,t}(C)$, define $[J,J']_1\in \A^{c,t}(C)$ by
\[
[J,J']_1 = \sum 
\begin{tikzpicture}[scale=0.4, baseline={(0,0.2)}, densely dashed]
\fill[fill=gray!50] (-2,2) circle [radius=1];
\fill[fill=gray!50] (2,2) circle [radius=1];
\node at (-2,2) {\small $J$};
\node at (2,2) {\small $J'$};
\draw (-1,1.5) -- (0,1);
\draw (1,1.5) -- (0,1);
\draw (0,1) -- (0,0) node[anchor=north] {\small $c_i$};
\end{tikzpicture}\ ,
\]
where the sum runs over all ways of connecting a univalent vertex in $J$ colored by $c_i$ and that in $J'$ for each $c_i\in C$.
Then, by definition, we obtain the next lemma.

\begin{lemma}
\label{lem:connect}
For $J,J'\in \A^{c,t}(\downarrow^C)$, the difference $J\cdot J'-J'\cdot J$ lies in $\A^{c,t}(\downarrow^C)$ and $\tau(J\cdot J'-J'\cdot J) = [\tau(J), \tau(J')]_1$ holds.
\end{lemma}

For a Jacobi diagram $J\in \A(C)$, let $U(J)$ denote the set of univalent vertices of $J$ and let $\ell(u)$ be the color (label) of $u\in U(J)$.
We recall from \cite[Section~3]{Lev02}
that the map $\eta\colon \A^{c,t}(C)\to \Q{C}\otimes L(C)$ defined by $\sum_{u\in U(J)} \ell(u)\otimes J_u$, where $J_u$ is the binary tree obtained from $J$ by declaring $u$ a root $\ast$.
Here $J_u$ is regarded as an element of the free Lie algebra $L(C)$ over $\Q$ generated by $C$ as follows:
\[
\begin{tikzpicture}[scale=0.4, baseline={(0,-0)}, densely dashed]
\draw (0,0) -- (-2,2) node[anchor=south] {$c_1$};
\draw (-1,1) -- (0,2) node[anchor=south] {$c_2$};
\draw (0,0) -- (2,2) node[anchor=south] {$c_3$};
\draw (0,0) -- (0,-1) node[anchor=north] {$\ast$};
\end{tikzpicture}
\quad\leftrightarrow\quad
[[c_1,c_2],c_3].
\]

\begin{proof}[Proof of Proposition~\ref{prop:J3J2n+1}]
Let $C=\{a,b\}$ and let $\pr_b\otimes\pr_{2n+2,2}$ denote the projection from $\Q{C}\otimes L(C)$ to the direct summand generated by elements of the form $b\otimes x$, where a Lie monomial $x\in L(C)$ consists of exactly $2n+2$ copies of $a$ and two $b$.
Using Lemma~\ref{lem:connect}, we show that
\begin{align*}
& (\pr_b\otimes\pr_{2n+2,2})\circ\eta(\tau(J_3\cdot J_{2n+1}-J_{2n+1}\cdot J_3)) \\
&= (\pr_b\otimes\pr_{2n+2,2})\circ\eta\biggl(4\cdot 
\begin{tikzpicture}[scale=0.4, baseline={(0,0.5)}, densely dashed]
\draw (0,3) -- (0,0) node[at start, anchor=south]{$a$} node[at end, anchor=north]{$b$};
\draw (2,0) -- (2,1) node[at start, anchor=north]{$b$};
\draw (0,2) -- (1.1,2);
\draw (2.9,2) -- (5,2);
\node at (4.1,2.7){$\cdots$};
\draw (5,3) -- (5,0) node[at start, anchor=south]{$a$} node[at end, anchor=north]{$b$};
\draw (1,3) .. controls +(0,-1) and +(-1,1) .. (2,1) node[at start, anchor=south]{$a$};
\draw (3,3) .. controls +(0,-1) and +(1,1) .. (2,1) node[at start, anchor=south]{$a$};
\end{tikzpicture}
\ \biggr) \\
&= 
4b \otimes \biggl(
-
\begin{tikzpicture}[scale=0.4, baseline={(0,0.3)}, densely dashed]
\draw (0,2) -- (0,0) node[at start, anchor=south]{$a$} node[at end, anchor=north]{$\ast$};
\draw (1,2) -- (1,1) node[at start, anchor=south]{$a$};
\draw (2,2) -- (2,1) node[at start, anchor=south]{$b$};
\draw (3,2) -- (3,1) node[at start, anchor=south]{$a$};
\node at (4,1.7){$\cdots$};
\draw (5,2) -- (5,0) node[at start, anchor=south]{$a$} node[at end, anchor=north]{$b$};
\draw (0,1) -- (5,1);
\end{tikzpicture}
+
\begin{tikzpicture}[scale=0.4, baseline={(0,0.3)}, densely dashed]
\draw (0,2) -- (0,0) node[at start, anchor=south]{$a$} node[at end, anchor=north]{$b$};
\draw (1,2) -- (1,1) node[at start, anchor=south]{$a$};
\draw (2,0) -- (2,1) node[at start, anchor=north]{$\ast$};
\draw (3,2) -- (3,1) node[at start, anchor=south]{$a$};
\node at (4,1.7){$\cdots$};
\draw (5,2) -- (5,0) node[at start, anchor=south]{$a$} node[at end, anchor=north]{$b$};
\draw (0,1) -- (5,1);
\end{tikzpicture}
-
\begin{tikzpicture}[scale=0.4, baseline={(0,0.3)}, densely dashed]
\draw (0,2) -- (0,0) node[at start, anchor=south]{$a$} node[at end, anchor=north]{$b$};
\draw (1,2) -- (1,1) node[at start, anchor=south]{$a$};
\draw (2,2) -- (2,1) node[at start, anchor=south]{$b$};
\draw (3,2) -- (3,1) node[at start, anchor=south]{$a$};
\node at (4,1.7){$\cdots$};
\draw (5,2) -- (5,0) node[at start, anchor=south]{$a$} node[at end, anchor=north]{$\ast$};
\draw (0,1) -- (5,1);
\end{tikzpicture}
\biggr).
\end{align*}
Let us express it by the Hall basis $\{b,a,[b,a],[b,[b,a]],[a,[b,a]],\ldots\}$.
The second term is already of this form up to sign.
Using the IHX relation, one can express the sum of the first and third terms as a sum of $2n-2$ terms:
\[
\begin{tikzpicture}[scale=0.4, baseline={(0,0.3)}, densely dashed]
\draw (0,2) -- (0,0) node[at start, anchor=south]{$a$} node[at end, anchor=north]{$\ast$};
\draw (1,2) -- (1,1) node[at start, anchor=south]{$a$};
\draw (3,1) -- (3,2);
\draw (2,3) -- (3,2) node[at start, anchor=south]{$a$};
\draw (4,3) -- (3,2) node[at start, anchor=south]{$b$};
\draw (5,2) -- (5,1) node[at start, anchor=south]{$a$};
\node at (6,1.7){$\cdots$};
\draw (7,2) -- (7,1) node[at start, anchor=south]{$a$};
\draw (8,2) -- (8,1) node[at start, anchor=south]{$a$};
\draw (0,1) -- (9,1) node[at end, anchor=south]{$b$};
\end{tikzpicture}
+\cdots+
\begin{tikzpicture}[scale=0.4, baseline={(0,0.3)}, densely dashed]
\draw (0,2) -- (0,0) node[at start, anchor=south]{$a$} node[at end, anchor=north]{$\ast$};
\draw (1,2) -- (1,1) node[at start, anchor=south]{$a$};
\node at (2,1.7){$\cdots$};
\draw (3,2) -- (3,1) node[at start, anchor=south]{$a$};
\draw (5,1) -- (5,2);
\draw (4,3) -- (5,2) node[at start, anchor=south]{$a$};
\draw (6,3) -- (5,2) node[at start, anchor=south]{$b$};
\draw (7,2) -- (7,1) node[at start, anchor=south]{$a$};
\draw (8,2) -- (8,1) node[at start, anchor=south]{$a$};
\draw (0,1) -- (9,1) node[at end, anchor=south]{$b$};
\end{tikzpicture}
\ .
\]
By the IHX relation again, it is equal to
\[
(2n-2)
\begin{tikzpicture}[scale=0.4, baseline={(0,0.3)}, densely dashed]
\draw (-2,1) -- (0,1) node[at start, anchor=south]{$a$};
\draw (-1,2) -- (-1,1) node[at start, anchor=south]{$b$};
\draw (0,0) -- (0,1) node[at start, anchor=north]{$\ast$};
\draw (1,2) -- (1,1) node[at start, anchor=south]{$a$};
\draw (2,2) -- (2,1) node[at start, anchor=south]{$a$};
\node at (3,1.7){$\cdots$};
\draw (4,2) -- (4,1) node[at start, anchor=south]{$a$};
\draw (5,2) -- (5,1) node[at start, anchor=south]{$a$};
\draw (0,1) -- (6,1) node[at end, anchor=south]{$b$};
\end{tikzpicture}
+
\sum
\begin{tikzpicture}[scale=0.4, baseline={(0,0.3)}, densely dashed]
\draw (-4,1) -- (0,1) node[at start, anchor=south]{$b$};
\draw (-3,0) -- (-3,1) node[at start, anchor=north]{$a$};
\node at (-2,0.3){$\cdots$};
\draw (-1,0) -- (-1,1) node[at start, anchor=north]{$a$};
\draw (0,0) -- (0,1) node[at start, anchor=north]{$\ast$};
\draw (1,2) -- (1,1) node[at start, anchor=south]{$a$};
\node at (2,1.7){$\cdots$};
\draw (3,2) -- (3,1) node[at start, anchor=south]{$a$};
\draw (4,2) -- (4,1) node[at start, anchor=south]{$a$};
\draw (0,1) -- (5,1) node[at end, anchor=south]{$b$};
\end{tikzpicture},
\]
where the second term is a sum of diagrams having at least two $a$ on the left bottom.\footnote{A precise expansion with respect to the Hall basis can be obtained by \texttt{FreeLieAlgebra} module provided in \texttt{Sage} (\cite{sagemath}) for small $n$.}
Since all the terms are parts of the Hall basis and $2n-2\neq 0$, we conclude that $(\pr_b\otimes\pr_{2n+2,2})\circ\eta(\tau(J_3\cdot J_{2n+1}-J_{2n+1}\cdot J_3))\neq 0$.
\end{proof}

\begin{proof}[Proof of Theorem~\ref{thm:deg_8}]
We first recall from \cite[Proposition~13.1]{HaMa00} or \cite[Proposition~E.24]{Oht02} that there exists a $2$-component string link $T_n$ satisfying $Z_\Phi(T_n)_{\leq 2n+1} = 1+J_{2n+1}$ for any associator $\Phi\in \widehat{\A}(\downarrow\downarrow\downarrow)_{\K}$.

For $\alpha\in \K\setminus\{0\}$, let $\Phi'$ be the associator obtained from $\Phi$ by twisting via an element $F= 1 + \alpha J_3 +(\deg\geq 4) \in \widehat{\A}(\downarrow\downarrow)_{\K}$.
Then, $F\cdot Z_{\Phi}(T_n) = Z_{\Phi'}(T_n)\cdot F$.
By comparing both sides in degree $2n+3$, we have $Z_{\Phi}(T_n)_{2n+3} = Z_{\Phi'}(T_n)_{2n+3}$.
On the other hand, by comparing them in degree $2n+4$, we have
\[
Z_{\Phi}(T_n)_{2n+4} +\alpha J_3\cdot J_{2n+1} = Z_{\Phi'}(T_n)_{2n+4} +\alpha J_{2n+1}\cdot J_3.
\]
Since $J_3\cdot J_{2n+1}-J_{2n+1}\cdot J_3\neq 0$ by Proposition~\ref{prop:J3J2n+1}, we conclude that $Z_{\Phi}(T_n)_{2n+4} \neq Z_{\Phi'}(T_n)_{2n+4}$.

When $\Phi$ is horizontal, we can choose $\Phi'$ to be horizontal as follows. Suppose that there is an element $\sigma \in \GRT_1(\K)$ such that 
\begin{align}
    \Phi' = \sigma \bullet \Phi \equiv \Phi + 2 \alpha [A+B, [A,B]] \bmod (\deg\geq  4).
    \label{eq:twistor}
\end{align}
Then, by Lemma~\ref{lem:twistor}, $\Phi'$ is obtained by twisting $\Phi$ via a twistor $F(\sigma) \in \widehat{\A}(\downarrow\downarrow)_{\K}$ such that $F(\sigma)= 1 + \alpha J_3 +(\deg\geq 4)$, and the same line of the above argument applies in this case as well. Thus, to prove the claim, it suffices to show the existence of such $\sigma\in \GRT_1(\K)$. In fact, there are many such $\sigma$.
For instance, we can take $\sigma = \mathrm{Exp}(2\alpha \sigma_3)$, where $\mathrm{Exp}\colon \grt_1(\K) \to \GRT_1(\K)$ is the exponential morphism (cf.~\cite[Chapter~14]{Mil17}), $\grt_1(\K)$ denotes the Lie algebra of $\GRT_1(\K)$, and $\sigma_3 \coloneqq [A+B,[A,B]]$. 
Recall from \cite[Proposition~6.3 and Remarks~(2) on page~859]{Dri91} that $\sigma_3$ lies in $\grt_1(\K)$ and spans its degree $3$ part. 
Also note that $\sigma=\mathrm{Exp}(2 \alpha \sigma_3) \equiv 1 + 2 \alpha \sigma_3 \bmod (\deg\geq  4)$ (see \cite[Section~5]{Dri91}). 
Hence, we have \eqref{eq:twistor} as desired and complete the proof.
\end{proof}

As mentioned in Remark~\ref{rem:deg_8}, when we apply Theorem~\ref{thm:deg_8} to an even rational horizontal associator, $\Phi'$ can be chosen as the KZ associator since they are related by twisting as in the proof of Theorem~\ref{thm:deg_8} (recall Example~\ref{ex:twisting}). 
Similarly, from the proof of Theorem~\ref{thm:deg_8}, one can observe that the $p$-adic KZ asscociator $\Phi_\KZ^p \in \GRT_1(\Q_p)$ introduced by Furusho in \cite{Fur04, Fur07} acts on $\widehat{\A}(\downarrow \downarrow)_{\Q_p}$ non-trivially, and therefore the corresponding action of an element of $\GT_1(\Q_p)$ on $\widehat{\Q_p[\SL_2]}$ as well, when the $p$-adic zeta value $\zeta_p(3)$ does not vanish. For which $p$ the condition $\zeta_p(3) \neq 0$ is satisfied, see \cite[Example~2.19, Remark~2.20]{Fur04} and references therein.

\begin{proof}[Proof of Proposition~\ref{prop:general_replacement}]
As in \cite[Exercise in Section~10.4.2]{CDM12}, we consider the twisting by
\[
F = 1 + \ \begin{tikzpicture}[scale=0.4, baseline={(0,0.5)}, densely dashed]
    \draw[solid,->] (0,3) -- (0,0);
    \draw[solid,->] (2,3) -- (2,0);
    \draw (0,2) -- (2,1);
     \draw (0,1) -- (2,2);
\end{tikzpicture}\ ,
\]
which adds $2([t_{12}, t_{23}] + t_{12} t_{13} - t_{23} t_{13})$ to the degree $2$ term of an associator. Note that the twisting via $F$ does not preserve horizontal associators since their degree $2$ terms are uniquely determined as $\frac{1}{24}[t_{12}, t_{23}]$ (\cite[Theorem~1]{Fur10}). By setting $\Phi'$ as the associator obtained from an associator $\Phi$ via $F$, the equality $Z_{\Phi'}(T)= F \cdot Z_{\Phi}(T) \cdot F^{-1} = Z_{\Phi}(T)$ holds by Corollary~\ref{cor:A_leq2_is_in_center}.
\end{proof}

\section{Lie algebra weight systems for the algebra $\A(\downarrow\downarrow)$}
\label{sec:weight_system}

This section gives a proof of Theorem~\ref{thm:weight_system}.
Let $\K$ be a field of characteristic $0$.

\subsection{Weight systems associated with standard representations of classical Lie algebras for $\A(\downarrow \downarrow)$}
Bar-Natan invented an elegant way for computing the weight system associated with the standard representation of a certain Lie algebra for chord diagrams (see \cite{Bar91}, \cite[Section~6.1]{CDM12} for the case of chord diagrams on the circle $S^1$). This section specializes in the weight systems for $\A(\downarrow \downarrow)$ and finds that their images form commutative subalgebras.

Let $N$ be a positive integer. 
Let $e_{i,j}$ denote the $N\times N$ matrix whose $(i,j)$-entry is $1$ and $0$ elsewhere $(i,j = 1, \ldots, N)$. Recall that the matrices $e_{i,j}$ form a basis of $\gl_N = \gl_N(\K)$. As in Section~\ref{subsec:Lie_alg_weight}, the trace form $B_0(x,y)=\tr(xy)$ is a symmetric non-degenerate ad-invariant bilinear form on $\gl_N$. We have  $B_0(e_{i,j}, e_{k,l})= \delta_{il}\delta_{jk}$, and hence $e_{i,j}^{\ast} = e_{j,i}$ with respect to $B_0$.

\begin{proposition}\label{prop:w_sys_glN}
    Let $B_0$ be the trace form on $\gl_N$ and $\St\colon \gl_N \to \End(\K^{N})$ the standard representation of $\gl_N$. Let $W_{\gl_N}^{\St}\colon \A(\downarrow\downarrow)_\K\to \End(\K^N)^{\otimes 2}$ denote the weight system associated with the triple $(\gl_N, B_0, \St)$.
    Then, for a Jacobi diagram $J \in \A(\downarrow\downarrow)_\K$ considered as a chord diagram, the following holds.
    \[
    W_{\gl_N}^{\St}(J) = \begin{cases}
        N^{s(J)-2}\, \id_{\K^N} \otimes \id_{\K^N} & (  \text{if $J$ is of type II}),\\
        N^{s(J) -2}\, \sum_{i,j=1}^N e_{i,j} \otimes e_{j,i} & (\text{if $J$ is of type X}),
    \end{cases}
    \]
    where $s(J)$ denotes the number of connected components of the curve obtained by doubling all chords of $J$ as follows:
    \[
    \begin{tikzpicture}[scale=0.5, baseline={(0,0.8)}, densely dashed]
    \draw[->, solid] (0,3) -- (0,1);
    \draw[<-, solid] (2,3) -- (2,1);
    \draw (0,2) -- (2,2);
    \end{tikzpicture}
    \quad
    \leftrightsquigarrow
    \quad 
    \begin{tikzpicture}[scale=0.5, baseline={(0,0.8)}, densely dashed]
    \draw[solid, rounded corners=8pt] (0,3) -- (0,2.2) -- (2,2.2) -- (2,3);
    \draw[solid, rounded corners=8pt] (2,1) -- (2,1.8) -- (0, 1.8) -- (0,1);
    \end{tikzpicture}\ .
    \]
    Here, $J$ is said to be of type II if non-closed curves obtained by the doubling are II-shaped, and of type X if they are X-shaped. 
\end{proposition}

\begin{proof}
The proof is similar to that of \cite[Theorem~6.7]{CDM12}, in which the weight system $W_{\gl_N}^{\St}$ for $\A(S^1)$ is considered. Hence, we only give a sketch of it. By the definition of $W_{\gl_N}^{\St}$, we assign the quadratic Casimir tensor $\mathbf{1}_{\gl_N} = \sum_{i,j=1}^N e_{i,j} \otimes e_{i,j}^{\ast} = \sum_{i,j=1}^N e_{i,j} \otimes e_{j,i}$ to each edge of $J$, and then we multiply the matrices associated with endpoints of the edges according to the orientations of the intervals. In addition, the standard basis $\{e_{i,j}\}$ satisfies the condition $e_{i,j} e_{k,l} = \delta_{jk} e_{il}$, which imposes a strong constraint on the indices in the products of the matrices along the intervals. Therefore, we can express how $e_{i,j} \otimes e_{j,i}$ is assigned to an edge as follows:
\[
    \begin{tikzpicture}[scale=0.5, baseline={(0,0.8)}, densely dashed]
    \draw[->, solid] (0,3) -- (0,1);
    \draw[<-, solid] (2,3) -- (2,1);
    \draw (0,2) -- (2,2);
    \node at (-0.9, 2) {$e_{i,j}$};
     \node at (2.9, 2) {$e_{j,i}$};
    \end{tikzpicture}
    \quad
    \leftrightsquigarrow
    \quad 
    \begin{tikzpicture}[scale=0.5, baseline={(0,0.8)}, densely dashed]
    \draw[->, solid] (0,3) -- (0,1);
    \draw[<-, solid] (2,3) -- (2,1);
    \draw (0,2) -- (2,2);
    \node at (-0.5, 2.7) {$i$};
     \node at (2.5, 2.7) {$i$};
     \node at (-0.5, 1.3) {$j$};
     \node at (2.5, 1.3) {$j$};
    \end{tikzpicture}
    \quad
    \leftrightsquigarrow
    \quad 
    \begin{tikzpicture}[scale=0.5, baseline={(0,0.8)}, densely dashed]
    \draw[solid, rounded corners=8pt] (0,3) -- (0,2.2) -- (2,2.2) -- (2,3);
    \draw[solid, rounded corners=8pt] (2,1) -- (2,1.8) -- (0, 1.8) -- (0,1);
    \node at (-0.7, 2.8) {$i$};
    \node at (-0.7, 1.3) {$j$};
    \end{tikzpicture}\ .
    \]
    From this diagrammatic expression, one sees that, for each state on the set of the edges of $J$ (recall Section~\ref{subsec:Lie_alg_weight}), the products of the matrices along intervals yield an operator either of the form $e_{i,i} \otimes e_{j,j}$ or $e_{i,j} \otimes e_{j,i}$ corresponding to the following diagrams respectively:
    \[
    \begin{tikzpicture}[scale=0.5, baseline={(0,0.8)}, densely dashed]
    \draw[->, solid] (0,3) -- (0,1);
    \draw[<-, solid] (2,3) -- (2,1);
    \draw (0,2) -- (2,2);
    \node at (-0.5, 2.7) {$i$};
     \node at (2.5, 2.7) {$j$};
     \node at (-0.5, 1.3) {$i$};
     \node at (2.5, 1.3) {$j$};
    \end{tikzpicture}
    \quad
    \leftrightsquigarrow
    \quad 
    \begin{tikzpicture}[scale=0.5, baseline={(0,0.8)}, densely dashed]
    \draw[solid] (0,3) -- (0,1);
    \draw[solid] (2,3) -- (2,1);
    \node at (-0.7, 2) {$i$};
    \node at (2.5, 2) {$j$};
    \end{tikzpicture}\ ,
    \quad 
    \begin{tikzpicture}[scale=0.5, baseline={(0,0.8)}, densely dashed]
    \draw[->, solid] (0,3) -- (0,1);
    \draw[<-, solid] (2,3) -- (2,1);
    \draw (0,2) -- (2,2);
    \node at (-0.5, 2.7) {$i$};
     \node at (2.5, 2.7) {$j$};
     \node at (-0.5, 1.3) {$j$};
     \node at (2.5, 1.3) {$i$};
    \end{tikzpicture}
    \quad
    \leftrightsquigarrow
    \quad 
    \begin{tikzpicture}[scale=0.5, baseline={(0,0.8)}, densely dashed]
    \draw[solid, rounded corners=8pt] (2,1) -- (0,3);
    \draw[solid, rounded corners=8pt] (0,1) -- (2,3);
     \node at (-0.7, 2.8) {$i$};
    \node at (-0.7, 1.3) {$j$};
    \end{tikzpicture}\ .
    \]
    To get $W_{\gl_N}^{\St}(J)$, we have to take the sum of all such operators over all possible indices.
    The number of indices corresponds to the number of values assigned to the indices $i,j,k, \ldots$ on the connected components of the curve obtained by the resolution procedure described above. Since each connected component admits exactly $N$ choices, excluding the sum corresponding to the two intervals, the total number is $N^{s(J)-2}$. Computing the sum corresponding to the two intervals, we obtain an element in either of the following two forms:
    \[
    \sum_{i,j} e_{i,i} \otimes e_{j,j} = \id_{\K^N} \otimes \id_{\K^N}, \quad \sum_{i,j} e_{i,j} \otimes e_{j,i} = \mathbf{1}_{\gl_N}.
    \]
    Thus, the assertion has been proved.
\end{proof}

\begin{figure}[h]
    \begin{tikzpicture}[scale=0.5, baseline={(0,0.8)}, densely dashed]
    \draw[solid] (0,3) -- (0,1);
    \draw[solid] (2,3) -- (2,1);
    \begin{scope}[xshift=5cm]
    \draw[solid, rounded corners=8pt] (2,1) -- (0,3);
    \draw[solid, rounded corners=8pt] (0,1) -- (2,3);
    \end{scope}
    \begin{scope}[xshift=10cm]
     \draw[solid, rounded corners=8pt] (0,3) -- (0,2.2) -- (2,2.2) -- (2,3);
    \draw[solid, rounded corners=8pt] (2,1) -- (2,1.8) -- (0, 1.8) -- (0,1);
    \end{scope}
    \end{tikzpicture}
    \caption{Three possible non-closed curves in the resolution of a Jacobi diagram in $\A(\downarrow \downarrow)$.}
    \label{fig:result_of_resolutions}
\end{figure}
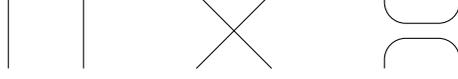

\begin{remark}
    The $\Ugraph$-shaped curve depicted in Figure~\ref{fig:result_of_resolutions} is not obtained as the result of the resolutions described in Proposition~\ref{prop:w_sys_glN}.
    However, when we consider $\A(\downarrow\uparrow)$ or $\A(\uparrow\downarrow)$, in contrast to the previous case, the $X$-shaped curve does not appear, while the $\Ugraph$-shaped one does.
\end{remark}

\begin{corollary}\label{cor:w_sys_im_gl_N}
    The image $\Im W_{\gl_N}^{\St} \subset \End(\K^N)^{\otimes 2}$ is commutative.
\end{corollary}

\begin{proof}
    By applying $W_{\gl_N}^{\St}$ to $\downarrow \downarrow$ and $t_{12} \in U(\mathfrak{t}_2)_\K \subset \A(\downarrow\downarrow)_\K$,
    one obtains $W_{\gl_N}^{\St}(\downarrow \downarrow)= \id_{\K^N} \otimes \id_{\K^N}$ and $W_{\gl_N}^{\St}(t_{12}) = \sum e_{i,j} \otimes e_{j,i}$, respectively. Therefore, by Proposition~\ref{prop:w_sys_glN}, $\Im W_{\gl_N}^{\St}$ is spanned by $W_{\gl_N}^{\St}(\downarrow \downarrow)$ and $W_{\gl_N}^{\St}(t_{12})$ and so $\Im W_{\gl_N}^{\St}$ forms a commutative subalgebra of $\End(\K^N)^{\otimes 2}$.
\end{proof}

Next, we consider the weight system associated with $\sl_N=\sl_N(\K)$ and the standard representation $\St \colon \sl_N \rightarrow \End(\K^N)$. As in \cite[Section~6.1.7]{CDM12}, we recall the notion of states for chord diagrams. For any Jacobi diagram $J \in \A(\downarrow \downarrow)_\K$, we take a representative of $J$ as a chord diagram. Then, a state $\sigma = \sigma_{\sl_N}^{\St}$ associated with $\sl_N$ and $\St$ is a map $\sigma\colon E_J \rightarrow \{1, -\frac{1}{N}\}$, where $E_J$ denotes the set of chords of (the chosen representative of) $J$.
For each state $\sigma$, we obtain an immersed curve from $J$ by applying the following resolution procedure to all its chords:
\[
    \begin{tikzpicture}[scale=0.5, baseline={(0,0.8)}, densely dashed]
    \draw[->, solid] (0,3) -- (0,1);
    \draw[<-, solid] (2,3) -- (2,1);
    \draw (0,2) -- (2,2);
    \node at (1, 2.5) {$a$};
    \end{tikzpicture}
    \quad
    \leftrightsquigarrow
    \quad 
    \begin{tikzpicture}[scale=0.5, baseline={(0,0.8)}, densely dashed]
    \draw[solid, rounded corners=8pt] (0,3) -- (0,2.2) -- (2,2.2) -- (2,3);
    \draw[solid, rounded corners=8pt] (2,1) -- (2,1.8) -- (0, 1.8) -- (0,1);
    \end{tikzpicture}\ , \text{if $\sigma(a)=1$;}
    \quad 
    \begin{tikzpicture}[scale=0.5, baseline={(0,0.8)}, densely dashed]
    \draw[->, solid] (0,3) -- (0,1);
    \draw[<-, solid] (2,3) -- (2,1);
    \draw (0,2) -- (2,2);
    \node at (1, 2.5) {$a$};
    \end{tikzpicture}
    \quad
    \leftrightsquigarrow
    \quad 
    \begin{tikzpicture}[scale=0.5, baseline={(0,0.8)}, densely dashed]
    \draw[solid, rounded corners=8pt] (0,1) -- (0,3);
    \draw[solid, rounded corners=8pt] (2,1) -- (2,3);
    \end{tikzpicture}\ , \text{if $\sigma(a)=-\frac{1}{N}$.}
    \]

Let $|\sigma|$ denote the number of connected components of the resulting curve. Then, the weight system associated with $\sl_N$ and $\St$ can be computed diagrammatically as follows. 

\begin{proposition}
Let $B_0$ be the trace form on $\sl_N$ and $\St\colon \sl_N \to \End(\K^{N})$ the standard representation of $\sl_N$. Let $W_{\sl_N}^{\St}\colon \A(\downarrow\downarrow)_\K \to \End(\K^N)^{\otimes 2}$ denote the weight system associated with the triple $(\sl_N, B_0, \St)$.
Then, for any Jacobi diagram $J \in \A(\downarrow\downarrow)_\K$ regarded as a chord diagram, we have
    \[
    W_{\sl_N}^{\St}(J) = 
        \sum_{\sigma} \left(\prod_a \sigma(a) \right)N^{|\sigma|-2} W_{J_{\sigma}},
    \]
    where the product runs over all chords of $J$, the sum runs over all states $\sigma=\sigma_{\sl_N}^{\St}$ for $J$. Here we define
    \[
    W_{J_{\sigma}} = \begin{cases}
        \id_{\K^N} \otimes \id_{\K^N}  & (\text{if $J_{\sigma}$ is of type II}),\\
        \sum_{i,j=1}^N e_{i,j} \otimes e_{j,i} & (\text{if $J_{\sigma}$ is of type X}),
    \end{cases}
    \]
    and $J_{\sigma}$ denotes the connected components of the curve obtained by the resolution procedure associated with $\sigma$ which are non-closed curves.
\end{proposition}

\begin{proof}
    It suffices to show the identity
    \begin{align*}
    \mathbf{1}_{\sl_N} &= \sum_{i,j=1}^N e_{i,j} \otimes e_{j,i} - \frac{1}{N} \id_{\K^N} \otimes \id_{\K^N}\\
    &= \mathbf{1}_{\gl_N}  - \frac{1}{N} \id_{\K^N} \otimes \id_{\K^N} \in \End(\K^N)^{\otimes 2}
    \end{align*}
    concerning the quadratic Casimir tensor with respect to $B_0$ because we can apply the same argument as Proposition~\ref{prop:w_sys_glN} once we grant it. 
    Although it may be known, we give a sketch of the proof to the above identity for the convenience to the reader. We take the basis of $\sl_N$ consisting of the matrices $e_{i,j}$ for $i \neq j$ and the matrices $h_i\coloneqq e_{i,i} - e_{i+1, i+1}$ $(i=1,2,\ldots, N-1)$ as in \cite[Exercise in Section~6]{CDM12}. We set
    \begin{align*}
    e_k  & \coloneqq \left(h_1 + 2 h_2 + \cdots + k h_k \right) \\
        & = \left(e_{1,1} + e_{2,2} + \cdots + e_{k,k} - k e_{k+1, k+1} \right)
    \end{align*}
    for $k=1,2,\ldots, N-1$. Then, we have $B_0( e_k, e_l) = k(1+k)\delta_{kl}$ for $k, l$ and $B_0( e_{i,j}, e_k ) =0$ for any $i\neq j$ and $k$. Therefore, the dual $e_k^{\ast}$ to $e_k$ with respect to $B_0$ is given by $e_k^{\ast}=\frac{1}{k(1+k)} e_k$. Notice that $e_{i,j}$ $(i\neq j)$ and $e_k$ form a basis of $\sl_N$. Therefore, the Casimir tensor can be written as
    \[
     \mathbf{1}_{\sl_N} = \sum_{i \neq j=1}^N e_{i,j} \otimes e_{j,i} + \sum_{k=1}^{N-1} e_k \otimes e_k^{\ast}.
    \]
    Then, the proof will be completed if we show that
    \[
    \sum_{k=1}^{N-1} e_k \otimes e_k^{\ast} = \sum_{i=1}^N e_{i,i} \otimes e_{i,i}  - \frac{1}{N} \id_{\K^N} \otimes \id_{\K^N},
    \]
    which can be checked by direct computation.
\end{proof}

\begin{corollary}\label{cor:w_sys_im_sl_N}
    The image $\Im W_{\sl_N}^{\St} \subset \End(\K^N)^{\otimes 2}$ is commutative.
\end{corollary}

\begin{proof}
    As in the case of $\Im W_{\gl_N}^{\St}$,  the image $\Im W_{\sl_N}^{\St}$ is also spanned by $W_{\gl_N}^{\St}(\downarrow \downarrow)$ and $W_{\gl_N}^{\St}(t_{12})$. Thus, the assertion follows.
\end{proof}

\begin{remark}
    Since $W_{\gl_N}^{\St}(t_{12})^2 = \id_{\K^N} \otimes \id_{\K^N}$, the images $\Im W_{\gl_N}^{\St}$ and $\Im W_{\sl_N}^{\St}$ coincide with  the subalgebra of $\End(\K^N)^{\otimes 2}$ isomorphic to the group ring $\K[\mathfrak{S}_2]$ of the symmetric group $\mathfrak{S}_2$ of order $2$.
\end{remark}

We then consider the weight system associated with $\so_N=\so_N(\K)$ and the standard representation $\St\colon \so_N \rightarrow \End(\K^N)$. As in the case of $\sl_N$, for any $J \in \A(\downarrow \downarrow)_\K$ as a chord diagram, we define a state $\sigma=\sigma_{\so_N}^{\St}$ as a map $\sigma \colon E_J \rightarrow \{\frac{1}{2}, -\frac{1}{2}\}$ and the associated resolution of $J$ is defined as follows:
\[
    \begin{tikzpicture}[scale=0.5, baseline={(0,0.8)}, densely dashed]
    \draw[->, solid] (0,3) -- (0,1);
    \draw[<-, solid] (2,3) -- (2,1);
    \draw (0,2) -- (2,2);
    \node at (1, 2.5) {$a$};
    \end{tikzpicture}
    \quad
    \leftrightsquigarrow
    \quad 
    \begin{tikzpicture}[scale=0.5, baseline={(0,0.8)}, densely dashed]
    \draw[solid, rounded corners=8pt] (0,3) -- (0,2.2) -- (2,2.2) -- (2,3);
    \draw[solid, rounded corners=8pt] (2,1) -- (2,1.8) -- (0, 1.8) -- (0,1);
    \end{tikzpicture}\ , \text{if $\sigma(a)=\frac{1}{2}$;}
    \quad 
    \begin{tikzpicture}[scale=0.5, baseline={(0,0.8)}, densely dashed]
    \draw[->, solid] (0,3) -- (0,1);
    \draw[<-, solid] (2,3) -- (2,1);
    \draw (0,2) -- (2,2);
    \node at (1, 2.5) {$a$};
    \end{tikzpicture}
    \quad
    \leftrightsquigarrow
    \quad 
    \begin{tikzpicture}[scale=0.5, baseline={(0,0.8)}, densely dashed]
    \draw[solid, rounded corners=8pt] (0,3) -- (0,2.5) -- (2,1.5) -- (2,1);
    \draw[solid, rounded corners=8pt] (2,3) -- (2,2.5) -- (0,1.5) -- (0,1);
    \end{tikzpicture}\ , \text{if $\sigma(a)=-\frac{1}{2}$.}
    \]

As in the case of $\sl_N$, we write $|\sigma|$ for the number of connected components of the curve obtained in this way.
Then, the diagrammatic computation of the weight system associated with $\so_N$ and $\St$ is described as follows.

\begin{proposition}\label{prop:w_sys_soN}
Let $B_0$ be the trace form on $\so_N$ and $\St\colon \so_N \to \End(\K^{N})$  the standard representation of $\so_N$. 
Let $W_{\so_N}^{\St}\colon \A(\downarrow\downarrow)_\K\to \End(\K^N)^{\otimes 2}$ denote the weight system associated with the triple $(\so_N, B_0, \St)$.
Then, for any Jacobi diagram $J \in \A(\downarrow\downarrow)_\K$ considered as a chord diagram, we have
    \[
    W_{\so_N}^{\St}(J) = 
        \sum_{\sigma} \left(\prod_a \sigma(a) \right)N^{|\sigma|-2} W_{J_{\sigma}},
    \]
    where the product runs over all chords of $J$, the sum runs over all states $\sigma=\sigma_{\so_N}^{\St}$ for  $J$,  we set
    \[
    W_{J_{\sigma}} = \begin{cases}
        \id_{\K^N} \otimes \id_{\K^N}  & (\text{if $J_{\sigma}$ is of type II}),\\
        \sum_{i,j=1}^N e_{i,j} \otimes e_{j,i} & (\text{if $J_{\sigma}$ is of type X}),\\
        \sum_{i,j=1}^N e_{i,j} \otimes e_{i,j} & (\text{otherwise}),
    \end{cases}
    \]
    and $J_{\sigma}$ denotes the connected components of the curve obtained by the resolution procedure associated with $\sigma$ which are non-closed curves.
    
\end{proposition}

\begin{proof}
By taking the basis of $\so_N$ consists of $e_{i,j} - e_{j,i}$ $(i<j)$, the quadratic Casimir tensor with respect to trace form $B_0$ can be expressed as 
    \begin{align*}
     \mathbf{1}_{\so_N} & = \sum_{i <j} \frac{1}{2}(e_{i,j}-e_{j,i}) \otimes (e_{j,i} - e_{i,j})\\
    &= \frac{1}{2} \sum_{i,j=1}^N e_{i,j} \otimes e_{j,i} - \frac{1}{2} \sum_{i,j=1}^N e_{i,j} \otimes e_{i,j}.
    \end{align*}
    Using this quadratic Casimir tensor, the assertion follows from the similar argument as Proposition~\ref{prop:w_sys_glN}.
\end{proof}

\begin{corollary}\label{cor:w_sys_im_so_N}
    The image $\Im W_{\so_N}^{\St} \subset \End(\K^N)^{\otimes 2}$ is commutative.
\end{corollary}

\begin{proof}
    By Proposition~\ref{prop:w_sys_soN}, the image $\Im W_{\so_N}^{\St}$ is spanned by $\id_{\K^N} \otimes \id_{\K^N}$, $\sum_{i,j=1}^N e_{i,j} \otimes e_{j,i}$ and $\sum_{i,j=1}^N e_{i,j} \otimes e_{i,j}$.  Since $\sum_{i,j=1}^N e_{i,j} \otimes e_{j,i}$ and $\sum_{i,j=1}^N e_{i,j} \otimes e_{i,j}$ commutes, we conclude that  the image $\Im W_{\so_N}^{\St}$ is commutative as desired.
\end{proof}

\begin{proposition}\label{prop:w_sys_sp2N}
Let $\sp_{2N}=\sp_{2N}(\K)$, $B_0$ denote the trace form on $\sp_{2N}$, and $\St\colon \sp_{2N} \to \End(\K^{2N})$ denote the standard representation of $\sp_{2N}$. 
Let $W_{\sp_{2N}}^{\St}\colon \A(\downarrow\downarrow)_\K\to \End(\K^{2N})^{\otimes 2}$ be the weight system associated with the triple $(\sp_{2N}, B_0, \St)$.
    Then, the image $\Im W_{\sp_{2N}}^{\St} \subset \End(\K^{2N})^{\otimes 2}$ is commutative.
\end{proposition}

\begin{proof}
    Let us take the basis of $\sp_{2N}$ which consists of 
    \begin{align*}
        & e_{i+N,j} + e_{j+N,i} \quad (1 \leq i< j \leq N), \\
        & e_{i+N,j} - e_{j-N,i} \quad  (1 \leq i \leq N, N+1 \leq j \leq 2N),\\
        & e_{i-N,j} + e_{j-N,i} \quad (N+1 \leq i< j \leq 2N),\\
        & e_{i+N, i}, \quad e_{i, i+N} \quad (1 \leq i \leq N). 
\end{align*}
Then, with a little computation, the quadratic Casimir tensor $\mathbf{1}_{\sp_{2N}}$ of $\sp_{2N}$ with respect to the trace form $B_0$ can be expressed as 
\begin{align*}
    \mathbf{1}_{\sp_{2N}}&= \frac{1}{2} \sum_{1 \leq i, j \leq N} (e_{i+N,j} \otimes e_{j,i+N} + e_{i+N,j} \otimes e_{i, j+N}) \\
    &+ \frac{1}{2}  \sum_{1 \leq i, j \leq N}(e_{i,j+N} \otimes e_{j+N, i} + e_{i,j+N} \otimes  e_{i+N,j}) \\
    &+ \frac{1}{2} \sum_{1 \leq i,j \leq N}( e_{i+N,j+N} \otimes e_{j+N,i+N} + e_{i,j} \otimes e_{j,i})\\
    &- \frac{1}{2} \sum_{1 \leq i,j \leq N}( e_{i+N,j+N} \otimes e_{i,j} + e_{i, j} \otimes e_{i+N,j+N}).
\end{align*}
Then, from this expression of $ \mathbf{1}_{\sp_{2N}}$, the computation of $ W_{\sp_{2N}}^{\St}$ reduces to the following diagrammatic one as in \cite[Page~17]{Bar91}. We define a state $\sigma=\sigma_{\sp_{2N}}^{\St}$ of a chord diagram $J$ in $\A(\downarrow \downarrow)_\K$ as a map from the set of all arcs of the intervals divided by univalent vertices to the set $\{P, Q\}$ consisting of symbols $P$ and $Q$. Then, for each state $\sigma$, we consider the corresponding resolution of $J$ as follows:
\begin{align*}
     & \begin{tikzpicture}[scale=0.5, baseline={(0,0.9)}, densely dashed]
    \draw[->, solid] (0,3) -- (0,1);
    \draw[<-, solid] (2,3) -- (2,1);
    \draw (0,2) -- (2,2);
    \node at (-0.5, 2.7) {$P$};
    \node at (-0.5, 1.3) {$P$};
    \node at (2.5, 2.7) {$P$};
    \node at (2.5, 1.3) {$P$};
    \end{tikzpicture}
    \leftrightsquigarrow
    \quad 
    \frac{1}{2}\ 
    \begin{tikzpicture}[scale=0.5, baseline={(0,0.9)}, densely dashed]
    \draw[solid, rounded corners=8pt] (0,3) -- (0,2.2) -- (2,2.2) -- (2,3);
    \draw[solid, rounded corners=8pt] (2,1) -- (2,1.8) -- (0, 1.8) -- (0,1);
     \node at (-0.5, 2.7) {$P$};
    \node at (-0.5, 1.3) {$P$};
    \node at (2.5, 2.7) {$P$};
    \node at (2.5, 1.3) {$P$};
    \end{tikzpicture}
    ,\quad 
    \begin{tikzpicture}[scale=0.5, baseline={(0,0.9)}, densely dashed]
    \draw[->, solid] (0,3) -- (0,1);
    \draw[<-, solid] (2,3) -- (2,1);
    \draw (0,2) -- (2,2);
    \node at (-0.5, 2.7) {$Q$};
    \node at (-0.5, 1.3) {$Q$};
    \node at (2.5, 2.7) {$Q$};
    \node at (2.5, 1.3) {$Q$};
    \end{tikzpicture}
    \leftrightsquigarrow
    \quad 
    \frac{1}{2}\ 
    \begin{tikzpicture}[scale=0.5, baseline={(0,0.9)}, densely dashed]
    \draw[solid, rounded corners=8pt] (0,3) -- (0,2.2) -- (2,2.2) -- (2,3);
    \draw[solid, rounded corners=8pt] (2,1) -- (2,1.8) -- (0, 1.8) -- (0,1);
    \node at (-0.5, 2.7) {$Q$};
    \node at (-0.5, 1.3) {$Q$};
    \node at (2.5, 2.7) {$Q$};
    \node at (2.5, 1.3) {$Q$};
    \end{tikzpicture}
    , \\
     & \begin{tikzpicture}[scale=0.5, baseline={(0,0.9)}, densely dashed]
    \draw[->, solid] (0,3) -- (0,1);
    \draw[<-, solid] (2,3) -- (2,1);
    \draw (0,2) -- (2,2);
    \node at (-0.5, 2.7) {$Q$};
    \node at (-0.5, 1.3) {$Q$};
    \node at (2.5, 2.7) {$P$};
    \node at (2.5, 1.3) {$P$};
    \end{tikzpicture}
     \leftrightsquigarrow
    \quad 
    - \frac{1}{2}\ 
    \begin{tikzpicture}[scale=0.5, baseline={(0,0.9)}, densely dashed]
    \draw[solid] (0,1) -- (2,3);
    \draw[solid] (2,1) -- (0,3);
    \node at (-0.5, 2.7) {$Q$};
    \node at (-0.5, 1.3) {$Q$};
    \node at (2.5, 2.7) {$P$};
    \node at (2.5, 1.3) {$P$};
    \end{tikzpicture}
    ,\quad 
    \begin{tikzpicture}[scale=0.5, baseline={(0,0.9)}, densely dashed]
    \draw[->, solid] (0,3) -- (0,1);
    \draw[<-, solid] (2,3) -- (2,1);
    \draw (0,2) -- (2,2);
    \node at (-0.5, 2.7) {$P$};
    \node at (-0.5, 1.3) {$P$};
    \node at (2.5, 2.7) {$Q$};
    \node at (2.5, 1.3) {$Q$};
    \end{tikzpicture}
    \leftrightsquigarrow
    \quad 
    - \frac{1}{2}\ 
    \begin{tikzpicture}[scale=0.5, baseline={(0,0.9)}, densely dashed]
    \draw[solid] (0,1) -- (2,3);
    \draw[solid] (2,1) -- (0,3);
    \node at (-0.5, 2.7) {$P$};
    \node at (-0.5, 1.3) {$P$};
    \node at (2.5, 2.7) {$Q$};
    \node at (2.5, 1.3) {$Q$};
    \end{tikzpicture}
    ,\\
    & \begin{tikzpicture}[scale=0.5, baseline={(0,0.9)}, densely dashed]
    \draw[->, solid] (0,3) -- (0,1);
    \draw[<-, solid] (2,3) -- (2,1);
    \draw (0,2) -- (2,2);
    \node at (-0.5, 2.7) {$P$};
    \node at (-0.5, 1.3) {$Q$};
    \node at (2.5, 2.7) {$P$};
    \node at (2.5, 1.3) {$Q$};
    \end{tikzpicture}
    \leftrightsquigarrow
    \quad 
    \frac{1}{2}\ 
    \Biggl(\ 
    \begin{tikzpicture}[scale=0.5, baseline={(0,0.9)}, densely dashed]
    \draw[solid, rounded corners=8pt] (0,3) -- (0,2.2) -- (2,2.2) -- (2,3);
    \draw[solid, rounded corners=8pt] (2,1) -- (2,1.8) -- (0, 1.8) -- (0,1);
    \node at (-0.5, 2.7) {$P$};
    \node at (-0.5, 1.3) {$Q$};
    \node at (2.5, 2.7) {$P$};
    \node at (2.5, 1.3) {$Q$};
    \end{tikzpicture} 
    + 
    \begin{tikzpicture}[scale=0.5, baseline={(0,0.9)}, densely dashed]
    \draw[solid] (0,1) -- (2,3);
    \draw[solid] (2,1) -- (0,3);
    \node at (-0.5, 2.7) {$P$};
    \node at (-0.5, 1.3) {$Q$};
    \node at (2.5, 2.7) {$P$};
    \node at (2.5, 1.3) {$Q$};
    \end{tikzpicture}
    \ \Biggr)
    ,\\
     & \begin{tikzpicture}[scale=0.5, baseline={(0,0.9)}, densely dashed]
    \draw[->, solid] (0,3) -- (0,1);
    \draw[<-, solid] (2,3) -- (2,1);
    \draw (0,2) -- (2,2);
    \node at (-0.5, 2.7) {$Q$};
    \node at (-0.5, 1.3) {$P$};
    \node at (2.5, 2.7) {$Q$};
    \node at (2.5, 1.3) {$P$};
    \end{tikzpicture}
    \leftrightsquigarrow
    \quad 
    \frac{1}{2}\ 
    \Biggl(\ 
    \begin{tikzpicture}[scale=0.5, baseline={(0,0.9)}, densely dashed]
    \draw[solid, rounded corners=8pt] (0,3) -- (0,2.2) -- (2,2.2) -- (2,3);
    \draw[solid, rounded corners=8pt] (2,1) -- (2,1.8) -- (0, 1.8) -- (0,1);
   \node at (-0.5, 2.7) {$Q$};
    \node at (-0.5, 1.3) {$P$};
    \node at (2.5, 2.7) {$Q$};
    \node at (2.5, 1.3) {$P$};
    \end{tikzpicture} 
    + 
    \begin{tikzpicture}[scale=0.5, baseline={(0,0.9)}, densely dashed]
    \draw[solid] (0,1) -- (2,3);
    \draw[solid] (2,1) -- (0,3);
    \node at (-0.5, 2.7) {$Q$};
    \node at (-0.5, 1.3) {$P$};
    \node at (2.5, 2.7) {$Q$};
    \node at (2.5, 1.3) {$P$};
    \end{tikzpicture}
    \ \Biggr).
\end{align*}
Here, we may understand that the  arcs labeled $P$ as being indexed by $1, \ldots, N$ and those labeled $Q$ as being indexed by $N+1, \ldots, 2N$, respectively. Note that, different from \cite[Page~17]{Bar91},  the curves obtained by the resolution of $J$ also inherit the labeling information by the symbols $P$ and $Q$, but whether labels are present or not has no effect when counting the contributions from closed curves as in \cite[Page~17]{Bar91}.
   Therefore, as the result of this diagrammatic computation, we finally get a linear combination of operators of the forms
   \begin{align*}
   &\id_{\K^{2N}} \otimes \id_{\K^{2N}} , \quad \sum_{i,j=1}^N e_{i,j} \otimes e_{j,i}, \quad \sum_{i,j=1}^N e_{i+N,j+N} \otimes e_{j+N,i+N},\\
   & \sum_{i,j=1}^N e_{i,j+N} \otimes e_{j+N,i},\quad \sum_{i,j=1}^N e_{i+N,j} \otimes e_{j,i+N},\quad \sum_{i,j=1}^N e_{i,j} \otimes e_{i+N, j+N},\\
   & \sum_{i,j=1}^N  e_{i+N, j+N} \otimes e_{i,j},\quad \sum_{i,j=1}^N e_{i+N, j} \otimes e_{i,j+N},\quad \sum_{i,j=1}^N e_{i, j+N} \otimes e_{i+N,j}
   \end{align*}
and they  commutate each other. It implies that the image $\Im W_{\sp_{2N}}^{\St}$ is commutative as desired.
\end{proof}

\subsection{The universal $\sl_2$ weight system for $\A(\downarrow\downarrow)$}

In this subsection we explore the weight system in the case where $\g = \sl_2(\K)$ further. It will be shown that the associated weight system has a commutative image at the universal level as follows.

\begin{proposition}\label{prop:univ_sl2_w_sys}
Let $\sl_2 = \sl_2(\K)$ and $B_0$ the trace form on it. Let $W_{\sl_2}\colon \A(\downarrow\downarrow)_\K \to U(\sl_2)^{\otimes 2}$ denote the universal $\sl_2$ weight system associated with the metrized Lie algebra $(\sl_2, B_0)$.
Then, the image of $W_{\sl_2}$ is commutative.
\end{proposition}

\begin{proof}
Let $\A=\A(\downarrow\downarrow)_\K$ and $W=W_{\sl_2}$.
It suffices to show that $W(\A_{\leq d}) \subset Z(\Im W)$ for any $d\geq 1$, where $Z(\Im W)$ denotes the center of the image $\Im W$. We prove it by induction on $d$. For $d =2$, it holds since $\A_{\leq 2} \subset Z(\A)$ by Corollary~\ref{cor:A_leq2_is_in_center}. Assume that it is true for $d-1$. For $J \in \A_{d}$, there are the following three possibilities (see Figure~\ref{fig:diag_proof_univ_sl2}):
\begin{enumerate}[label=(\alph*)]
\item $J$ contains at least one $H$-shaped subgraph.
\item $J$ does not contain $H$-shaped subgraphs but contains at least one $Y$-shaped subgraph. 
\item $J$ contains neither $H$-shaped nor $Y$-shaped subgraphs, i.e., $J$ consists of only chords.
\end{enumerate}

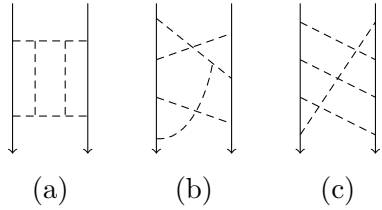
\begin{figure}[h]
    \begin{tikzpicture}[scale=0.5, densely dashed]
     \draw[solid, ->] (0, 2) -- (0, -2);
    \draw[solid, ->] (2, 2) -- (2, -2);
    \draw (0, 1) -- (2,1);
    \draw (0, -1) -- (2,-1);
    \draw (0.6, 1) -- (0.6,-1);
    \draw (1.4, 1) -- (1.4,-1);
    \node at (1,-3) {(a)};
    \end{tikzpicture}
    \quad \quad 
    \begin{tikzpicture}[scale=0.5, densely dashed]
        \draw[solid, ->] (0, 2) -- (0, -2);
    \draw[solid, ->] (2, 2) -- (2, -2);
       \draw (0, 1.6) -- (2,0);
       \draw (0,0.5) -- (2,1.2);
       \draw (0, -0.5) -- (2, -1.2);
       \draw (0, -1.6) .. controls (1, -1.7) and (1.45, 0) .. (1.5, 0.45);
       \node at (1,-3) {(b)};
    \end{tikzpicture}
    \quad \quad 
    \begin{tikzpicture}[scale=0.5, densely dashed]
        \draw[solid, ->] (0, 2) -- (0, -2);
    \draw[solid, ->] (2, 2) -- (2, -2);
       \draw (0, 1.5) -- (2,0.5);
       \draw (0,0.5) -- (2,-0.5);
       \draw (0, -0.5) -- (2, -1.5);
       \draw (0, -1.5) -- (2, 1.5);
       \node at (1,-3) {(c)};
    \end{tikzpicture}
    \caption{Typical examples of the cases (a), (b), and (c).}
    \label{fig:diag_proof_univ_sl2}
\end{figure}
For the case (a), the skein relation \eqref{eq:skein_sl2} implies that  $J \in \A_{d-1}(\downarrow\downarrow)$, and hence $W(J) \in Z(\Im W)$ holds  by the induction hypothesis. For the case (b), by iterated use of STU relation to univalent vertices of $Y$-shaped subgraphs and the adjacent univalent vertices, we obtain 

\[
\begin{tikzpicture}[scale=0.5, baseline={(0,0.8)}, densely dashed]
    \draw[solid] (0,4) -- (0,2.8);
    \draw[->, solid] (0,1.2) -- (0,0);
    \draw[solid] (2,4) -- (2,2.8);
    \draw[->, solid] (2,1.2) -- (2,0);
    \draw[solid] (-0.5, 1.2) -- (2.5, 1.2) -- (2.5, 2.8) -- (-0.5, 2.8) -- cycle;
    \node at (1, 2) {$J$};
\end{tikzpicture}
\quad 
= 
\quad 
\begin{tikzpicture}[scale=0.5, baseline={(0,0.8)}, densely dashed]
    \draw[ solid] (0,4) -- (0,3.6);
    \draw[->, solid] (0,2) -- (0,0);
    \draw[ solid] (2,4) -- (2,3.6);
    \draw[->, solid] (2,2) -- (2,0);
    \draw[solid] (-0.5, 2) -- (2.5, 2) -- (2.5, 3.6) -- (-0.5, 3.6) -- cycle;
    \draw (0, 1.5) -- (1, 1) -- (2,1);
    \draw (0, 0.5) -- (1, 1);
    \node at (1, 2.8) {$J'$};
\end{tikzpicture}
\quad 
+ 
\quad
\sum_{J''} 
\quad 
\begin{tikzpicture}[scale=0.5, baseline={(0,0.8)}, densely dashed]
    \draw[solid] (0,4) -- (0,2.8);
    \draw[->, solid] (0,1.2) -- (0,0);
    \draw[solid] (2,4) -- (2,2.8);
    \draw[->, solid] (2,1.2) -- (2,0);
    \draw[solid] (-0.5, 1.2) -- (2.5, 1.2) -- (2.5, 2.8) -- (-0.5, 2.8) -- cycle;
    \node at (1, 2) {$J''$};
\end{tikzpicture}\ ,
\]
where $J' \in \A_{d-2}$ and all the diagrams $J''$ in the sum contain at least one $H$-shaped subgraph. Note that in this equality we implicitly suppose that $Y$-shaped subgraph lies on both strings. However, the similar equality holds if the $Y$-shaped subgraph lies on only one of the two strings. Therefore, the induction hypothesis and the argument for the case (a) imply $W(J) \in Z(\Im W)$.
For the case (c), as in the case (b), iterated usages of STU relation to one of the chords in $J$ leads to the following equality:
\[
\begin{tikzpicture}[scale=0.5, baseline={(0,0.8)}, densely dashed]
    \draw[solid] (0,4) -- (0,2.8);
    \draw[->, solid] (0,1.2) -- (0,0);
    \draw[solid] (2,4) -- (2,2.8);
    \draw[->, solid] (2,1.2) -- (2,0);
    \draw[solid] (-0.5, 1.2) -- (2.5, 1.2) -- (2.5, 2.8) -- (-0.5, 2.8) -- cycle;
    \node at (1, 2) {$J$};
\end{tikzpicture}
\quad 
= 
\quad 
\begin{tikzpicture}[scale=0.5, baseline={(0,0.8)}, densely dashed]
    \draw[ solid] (0,4) -- (0,3.6);
    \draw[->, solid] (0,2) -- (0,0);
    \draw[ solid] (2,4) -- (2,3.6);
    \draw[->, solid] (2,2) -- (2,0);
    \draw[solid] (-0.5, 2) -- (2.5, 2) -- (2.5, 3.6) -- (-0.5, 3.6) -- cycle;
    \draw (0, 1)  -- (2,1);
    \node at (1, 2.8) {$J'$};
\end{tikzpicture}
\quad 
+ 
\quad
\sum_{J''} 
\quad 
\begin{tikzpicture}[scale=0.5, baseline={(0,0.8)}, densely dashed]
    \draw[solid] (0,4) -- (0,2.8);
    \draw[->, solid] (0,1.2) -- (0,0);
    \draw[solid] (2,4) -- (2,2.8);
    \draw[->, solid] (2,1.2) -- (2,0);
    \draw[solid] (-0.5, 1.2) -- (2.5, 1.2) -- (2.5, 2.8) -- (-0.5, 2.8) -- cycle;
    \node at (1, 2) {$J''$};
\end{tikzpicture}\ ,
\]
where $J' \in \mathcal{A}_{d-1}$ and the diagrams $J''$ in the sum consist of Jacobi diagrams of the case (b). Therefore, by the induction hypothesis and the consequence of the case (b), we conclude that $W(J) \in Z(\Im W)$. Again, note that, in the above equality, we depict the result of the application of STU relation to one of the chord in $J$ which lies on both strings but the similar equality can be obtained if we choose another chord in $J$ which lies on one of the two strings. Therefore, the assertion has been proved.
\end{proof}

\subsection{Proof of Theorem~\ref{thm:weight_system}}

Theorem~\ref{thm:weight_system} is a consequence of the following theorem. 
\begin{theorem}\label{thm:weight_system_precise}
    Let $N$ be a positive integer. For any associator $\Phi$ and a weight system $W \in \{W_{\sl_2}, W_{\gl_N}^{\St}, W_{\sl_N}^{\St}, W_{\so_N}^{\St},  W_{\sp_{2N}}^{\St}\}$ on two strands, the composite map $W \circ Z_{\Phi}$ does not depend on the choice of the associator $\Phi$.
\end{theorem}
\begin{proof}
    As recalled in Section~\ref{subsec:Kontsevich}, for associators $\Phi, \Phi'$ and any 2-component string link $T$, we have $Z_{\Phi'}(T) = F\cdot Z_{\Phi}(T)\cdot F^{-1}$ for some $F \in \A(\downarrow\downarrow)_\K$. Since $W$ is an algebra homomorphism, we have $W\circ Z_{\Phi'}(T) = W(F) (W\circ Z_{\Phi}(T)) W(F)^{-1}$. The assertion follows from this identity and Corollaries~\ref{cor:w_sys_im_gl_N}, \ref{cor:w_sys_im_sl_N}, \ref{cor:w_sys_im_so_N} and Propositions~\ref{prop:w_sys_sp2N}, \ref{prop:univ_sl2_w_sys}.
\end{proof}

\begin{remark}
    Theorem~\ref{thm:weight_system_precise} is also valid for $\A(\downarrow^{\epsilon})$, where $\epsilon = +-, -+, --$ with a slight modification to the arguments given in the previous subsections. More precisely, the only difference is that the dual representation is assigned to the interval with the letter $-$ and the rest is essentially the same as in the case of  $\A(\downarrow \downarrow)$.
\end{remark}

\appendix
\section{Computational results on $\grt_1$ and Drinfeld associators}

This appendix gives explicit forms of a basis of the Lie algebra $\grt_1(\K)$ in low degrees, obtained by a \texttt{SageMath} program.
The computational results are essentially due to Adrien Brochier~\cite{Broc25}.

The \emph{Grothendieck--Teichm\"uller Lie algebra} $\grt_1(\K)$ is the Lie algebra of $\GRT_1(\K)$ consisting of $\psi \in \mathfrak{f}(A,B)_\K$ such that 
\begin{align*}
    & \psi(B,A) = - \psi(A,B),\\
    & \psi(C,A) + \psi(B,C) + \psi(A,B) = 0\quad \text{for}\ A + B + C = 0,\\
    & \psi(t_{12}, t_{23} + t_{24}) + \psi(t_{13} + t_{23}, t_{34})\\
    &\qquad = \psi(t_{23}, t_{34}) + \psi(t_{12} + t_{13}, t_{24} + t_{34}) + \psi(t_{12}, t_{23})\quad \text{in}\ U(\mathfrak{t}_4)_\K.
\end{align*}
The Lie algebra structure of $\grt_1(\K)$ is given by the braket operation called the \emph{Ihara bracket} defined as follows: for $\psi_1, \psi_2 \in \grt_1(\K)$, 
\[
\{ \psi_1, \psi_2\} =   D_{\psi_1}(\psi_2)-  D_{\psi_2}(\psi_1)-[\psi_1, \psi_2],
\]
where $D_{\psi_i}$ $(i=1,2)$ is the derivation on $\mathfrak{f}(A,B)_\K$ determined by $D_{\psi_i}(A)=0$ and $D_{\psi_i}(B)=[\psi_i, B]$. As in \cite[Section~10.2.7]{CDM12}, for simplicity of notation, we set
\[
C_{kl} = \ad_B^{k-1} \ad_A^{l-1}[A,B]\quad (k, l \geq 1).
\]
Then, it can be checked that the degree $\leq 8$ part of $\grt_1(\K)$ is given by $\grt_1(\K)_{\leq 8} = \K \sigma_3 \oplus \K \sigma_5 \oplus \sigma_7 \oplus \K \{\sigma_3, \sigma_5 \}$, where 
\begin{align*}
& \sigma_3 = C_{12} + C_{21}, \\
& \sigma_5 = C_{14} + C_{41} + 2(C_{23} + C_{32}) + \frac{1}{2}[C_{11}, C_{12}] + \frac{3}{2} [C_{11}, C_{21}], \\
& \sigma_7 = C_{16} + C_{61} + 3 (C_{25} + C_{52}) + 5(C_{34} + C_{43}) + 4[C_{11}, C_{14}]+ 13[C_{11}, C_{23}]\\
&\qquad + 12[C_{11}, C_{32}] + 5 [C_{11}, C_{41}] + 3[C_{12}, C_{13}] + \frac{61}{16}[C_{12}, C_{22}]\\
&\qquad - \frac{19}{16} [C_{12}, C_{31}]  + \frac{99}{16}[C_{21}, C_{13}] +\frac{179}{16} [C_{21}, C_{22}] + 3 [C_{21},C_{31}]\\
&\qquad + \frac{65}{16}[C_{11}, [C_{11}, C_{12}]]+ \frac{17}{16}[C_{11}, [C_{11}, C_{21}]], \\
& \{ \sigma_3, \sigma_5\} = D_{\sigma_3}(\sigma_5) - D_{\sigma_5}(\sigma_3) -[\sigma_3, \sigma_5]\\
&=- 2[C_{11} ,  C_{15}] - 5  [  C_{11} ,  C_{24} ] + 5  [ C_{11} ,  C_{42}] + 2 [ C_{11} ,  C_{51} ] - 4  [ C_{11} ,  [ C_{11} , C_{13}]\\
&\qquad + 9  [ C_{11} ,  [  C_{11} ,  C_{22}]] + 6  [  C_{11},  [ C_{11} ,  C_{31} ]  ] -5  [ C_{12} ,  C_{14} ]-9  [ C_{12} , C_{23} ] \\
&\qquad + \frac{3}{2}  [  C_{12} ,  C_{32}] + \frac{7}{2}  [ C_{12},  C_{41}  ] - \frac{11}{2}  [ C_{12} ,  [ C_{11},  C_{12} ] ] + \frac{3}{2}  [  C_{12} ,  [  C_{11} , C_{21} ]  ] \\
&\qquad - \frac{7}{2}  [ C_{21} , C_{14} ] - \frac{3}{2}  [ C_{21},  C_{23}  ] + 9 [  C_{21} ,  C_{32} ] + 5  [ C_{21} , C_{41}  ] \\
&\qquad + \frac{7}{2}  [ C_{21} ,  [ C_{11} , C_{21}  ]  ] + \frac{3}{2}  [ C_{13} ,  C_{22}  ] + 3  [ C_{13} , C_{31}  ] + \frac{3}{2}  [ C_{22},  C_{31} ].
\end{align*}

The logarithm of the general rational Drinfeld associator $\Phi_{\Q} \in M_1(\Q)$ can be described as follows up to degree $5$:
\begin{align*}
   & \log \Phi_{\mathbb{Q}} = \frac{1}{24} C_{11} + \lambda_1 \sigma_3 - \frac{1}{1440} (C_{13} + C_{31}) - \frac{1}{5760} C_{22} + \lambda_2 \sigma_5\\
   &+ \frac{\lambda_1}{24} \left(C_{23} + C_{32} + \frac{1}{2} [C_{11}, C_{12}] + \frac{1}{2} [C_{11}, C_{21}]\right) + (\text{terms of degree $\geq 6$}).
\end{align*}
Here, $\lambda_1, \lambda_2 \in \mathbb{Q}$ are arbitrary constants.

\def\cprime{$'$}

\end{document}